\documentclass[a4paper,oneside,english]{amsart}
\usepackage[T1]{fontenc}
\usepackage[utf8]{inputenc}
\usepackage{verbatim}
\usepackage{amssymb}

\makeatletter
  \theoremstyle{remark}
  \newtheorem*{acknowledgement*}{Acknowledgement}
  \theoremstyle{plain}
  \newtheorem*{thm*}{Theorem}
  \theoremstyle{plain}
  \newtheorem{thm}{Theorem}
  \theoremstyle{plain}
  \newtheorem{lem}{Lemma}
 \theoremstyle{definition}
 \newtheorem*{defn*}{Definition}
  \theoremstyle{remark}
  \newtheorem*{rem*}{Remark}
\newenvironment{lyxlist}[1]
{\begin{list}{}
{\settowidth{\labelwidth}{#1}
 \setlength{\leftmargin}{\labelwidth}
 \addtolength{\leftmargin}{\labelsep}
 }}
{\end{list}}
  \theoremstyle{plain}
  \newtheorem{prop}{Proposition}

\usepackage{amscd}
\usepackage{graphicx}
\usepackage[hang,centerlast]{subfigure}
\numberwithin{equation}{section}
\numberwithin{thm}{section}
\numberwithin{lem}{section}
\numberwithin{prop}{section}
\DeclareMathOperator{\lspan}{span}
\def\mypsteps#1#2{#1}

\usepackage{babel}
\makeatother

\begin{document}

\title{Heat kernel estimates for the Gru{\v{s}}in operator}

\author{Martin Paulat}

\email{martin@inter-zone.de}

\date{07/30/07}

\begin{abstract}
We study the geometry associated to the Grušin operator \[
G=\Delta_{x}+\left|x\right|^{2}\partial_{u}^{2}\mbox{ \mbox{on }}\mathbb{R}_{x}^{n}\times\mathbb{R}_{u},\]
to obtain heat kernel estimates for this operator. The main work is
to find the shortest geodesics connecting two given points in $\mathbb{R}^{n+1}$.
This gives the Carnot-Carathéodory distance $d_{CC}$, associated
to this operator. The main result in the second part is to give Gaussian
bounds for the heat kernel $K_{t}$ in terms of the Carnot-Carathéodory
distance. In particular we obtain the following estimate\[
\left|k_{t}(\zeta,\eta)\right|\leq C\, t^{-\frac{n}{2}-1}\min\left(1+\frac{d_{CC}(\zeta,\eta)}{\left|x+\xi\right|},1+\frac{d_{CC}(\zeta,\eta)^{2}}{4t}\right)^{\alpha}e^{-\frac{1}{4t}d_{CC}(\zeta,\eta)^{2}}\]
for all $\zeta=(x,u_{1}),\;\eta=(\xi,u)\in\mathbb{R}^{n+1}$, where
$\alpha=\max\left\{ \frac{n}{2}-1,0\right\} $. Here the homogeneous
dimension is $q=n+2$, so that $\frac{n}{2}-1=\frac{q-4}{2}$. This
shows that our result for $n\geq2$ corresponds with the result on
the Heisenberg group, which was given by Beals, Gaveau, Greiner in
\cite{Beals2000}.
\end{abstract}
\maketitle
\begin{acknowledgement*}
I would like to express my gratitude to my advisor Professor Dr. Detlef
Müller for constant support and numberless helpful suggestions.
\end{acknowledgement*}

\section{Introduction}

The purpose of this article is to study the geometric properties of
the Grušin operator \[
G=\Delta_{x}+|x|^{2}\partial_{u}^{2}\quad\mbox{on}\:\mathbb{R}_{x}^{n}\times\mathbb{R}_{u}\]
to give estimates for the heat kernel of this operator. One may write
G as\[
G=\sum_{j=1}^{n}\left(X_{j}^{2}+U_{j}^{2}\right),\]
with the smooth vector fields \[
X_{j}:=\partial_{x_{j}},\qquad U_{j}:=x_{j}\partial_{u},\; j=1,\ldots,n.\]
Note that $G$ is hypoelliptic, since $G$ satisfies the Hörmander
condition. We are interested in the Carnot-Carathéodory distance associated
to these vector fields, i.e. the length of a minimizing horizontal
curve. A horizontal curve is a curve, whose tangents are a linear
combination of the vector fields $X_{j},\; U_{j},\; j=1,\ldots,n$.
Depending on the start and end point of the curve, there are a different
number of locally minimizing curves, i.e. geodesics. In addition to
the results of Calin, Chang, Greiner and Kannai in \cite{Calin2005},
we can compare their lengths to give explicit formulas for the Carnot-Carathéodory
distance. The geodesics starting in $0\in\mathbb{R}^{n+1}$ are very
similar to those on the Heisenberg group (see \cite{Beals2000}).
But comparing the lengths of geodesics starting in a generic point
$(x,0)\in\mathbb{R}^{n}\times\mathbb{R}$ gets more complicated. Another
difference occurs when studying curves connecting $(x,0)$ and $(\pm x,u)$,
where on gets two different types of geodesics (see Theorem \eqref{thm:x1=00003Dx}
and \eqref{thm:x1=00003D-x}).

These results on the Carnot-Carathéodory distance allow us to give
some pointwise estimates of the heat kernels. 

Due to the fact that the partial Fourier transform of $G$ in $u$
leads to the rescaled Hermite operators, it is easy to calculate the
heat kernel using Mehler's formula, up to the partial Fourier transform.
One may also observe that $G$ is translation invariant in the $u$-variable.
This allows us to write the heat kernel in the form $K_{t}(x,\xi,u),\; x,\xi\in\mathbb{R}^{n},\; u\in\mathbb{R}$.
So the solution to the heat equation\[
\begin{cases}
\left(\partial_{t}-G\right)u & =0\\
u(0,\cdot)=f\end{cases}\]
is given by\[
u(t,x)=\int_{\mathbb{R}^{n}}\int_{\mathbb{R}}K_{t}(x,\xi,u-\lambda)f(\xi,\lambda)d\lambda d\xi.\]
We are able to give global estimates for the kernel of the form\[
\left|K_{t}(x,\xi,u)\right|\leq F(x,\xi,u,t)e^{-\frac{1}{4t}d_{CC}\left((x,0),(\xi,u)\right)^{2}},\]
where $F$ is a function with polynomial growth. It turns out that\[
F(x,\xi,u,t)\leq C\, t^{-\frac{n}{2}-1}\min\left(1+\frac{d_{CC}(\zeta,\eta)}{\left|x+\xi\right|},1+\frac{d_{CC}(\zeta,\eta)^{2}}{4t}\right)^{\alpha},\]
with $\alpha=\max\left\{ \frac{n}{2}-1,0\right\} $ and $C>0$ is
a constant independent of $x,\xi,u,t$. In particular our main result
is the following

\begin{thm*}
For $\zeta=(x,0),\;\eta=(\xi,u)\in\mathbb{R}^{n+1}$ we have\begin{equation}
\left|K_{t}(x,\xi,u)\right|\lesssim t^{-\frac{n}{2}-1}\min\left(1+\frac{d_{CC}(\zeta,\eta)}{\left|x+\xi\right|},1+\frac{d_{CC}(\zeta,\eta)^{2}}{4t}\right)^{\alpha}e^{-\frac{1}{4t}d_{CC}(\zeta,\eta)^{2}},\end{equation}
with $\alpha=\max\left(\frac{n}{2}-1,0\right)$.
\end{thm*}
For a better understanding of the exponent $\alpha$, consider the
dilation\[
\delta_{r}(x,u):=(rx,r^{2}u),\quad x\in\mathbb{R}^{n},\; u\in\mathbb{R},\]
for $r>0.$ Then $G$ is homogeneous with degree 2, and also the Carnot-Carathéodory
distance is homogeneous with degree 2, i.e.\[
d_{CC}\left(\delta_{r}(\zeta),\delta_{r}(\eta)\right)=r^{2}d_{CC}(\zeta,\eta),\;\zeta,\eta\in\mathbb{R}^{n+1}.\]
In this setting the homogeneous dimension is $q=n+2$, and $\frac{n}{2}-1=\frac{q-4}{2}$,
$\frac{n}{2}+1=\frac{q}{2}$. This shows now that our result exactly
coincides, in the case $n\geq2$, with those on the Heisenberg group
obtained by Beals, Gaveau, Greiner in \cite{Beals2000}.

There are some Gaussian-type estimates for heat kernels with greater
generality. Sikora, for example, shows in \cite{Sikora2004} \[
\left|K_{t}(x,\xi,u)\right|\lesssim t^{-\frac{q}{2}}\left(1+\frac{d_{CC}(\zeta,\eta)}{4t}\right)^{\frac{q-1}{2}}e^{-\frac{d_{CC}(\zeta,\eta)^{2}}{4t}}.\]
The methods that are used are completely different. The explicit formulas
in our situation allow us to give a smaller exponent of the polynomial
factor, i.e. $\frac{q-4}{2}$ instead of $\frac{q-1}{2}$.

\section{The sub-Riemannian geometry associated to the Grušin operator\label{sec:Geometry-introduced-by}}

Let $M$ be a connected $C^{\infty}$-Manifold with smooth real vector
fields $X_{1},\ldots,X_{m}$. For $x\in M$ and $v\in T_{x}M$ define:\begin{equation}
\left\Vert v\right\Vert _{x}^{2}:=\inf\left\{ \sum_{j=1}^{m}\alpha_{j}^{2}:\; v=\sum_{j=1}^{m}\alpha_{j}X_{j}(x)\right\} ,\end{equation}
where we use the convention, that $\inf\varnothing:=\infty$, i.e.
if $v$ is not in the linear span of $X_{1},\ldots,X_{m}$, then we
set $\left\Vert v\right\Vert _{x}^{2}:=\infty$. A horizontal curve
$\gamma:\left[0,1\right]\rightarrow M$ is an absolutely continuous
curve such that $\dot{\gamma}(t)\in\lspan\left\{ X_{1}(\gamma(t)),\ldots,X_{m}(\gamma(t))\right\} $
for almost every $t\in\left[0,1\right]$. We define the length of
a horizontal curve by\begin{equation}
L(\gamma):=\int\limits _{0}^{1}\left\Vert \dot{\gamma}(t)\right\Vert _{\gamma(t)}dt.\end{equation}
Then the Carnot-Carathéodory distance of two points $p,q\in M$ associated
to $X_{1},\ldots X_{m}$ is defined by\begin{equation}
d_{CC}(p,q):=\inf L(\gamma),\end{equation}
where the infimum is taken over all horizontal curves $\gamma$ connecting
$p$ and $q$, i.e. $\gamma(0)=p$, $\gamma(1)=q$.

It is well known by Chow's theorem (see \cite{Chow1939}) that, if
$X_{1,}\ldots,X_{m}$ and their brackets span the tangent space $T_{x}M$
at every point $x$ of $M$, then any two points can be joined by
such a curve, so that $d_{CC}(p,q)<\infty$, for any $p,q\in M$. 

Now let $M=\mathbb{R}^{n+1}$ and \[
X_{j}=\partial_{x_{j}},\qquad U_{j}=x_{j}\partial_{u},\qquad j=1,\ldots,n.\]
Then the Grušin operator $G=\Delta_{x}+\left|x\right|^{2}\partial_{u}^{2}$
reads in terms of these vector fields \[
G=\sum\limits _{j=1}^{n}\left(X_{j}^{2}+U_{j}^{2}\right).\]
The vector fields $X_{1},\ldots,X_{n},U_{1},\ldots,U_{n}$ span the
tangent space everywhere except along the line $\left|x\right|=0.$
But, since $\left[X_{j},U_{j}\right]=\partial_{u}$, $j=1,\ldots,n$,
all conditions are fulfilled to connect any two points by an absolutely
continuous curve with finite length. If $\gamma:\left[0,1\right]\rightarrow\mathbb{R}^{n+1}$
is an absolutely continuous curve, one has\begin{alignat*}{1}
\dot{\gamma} & =\dot{\gamma}_{n+1}\partial_{u}+\sum_{j=1}^{n}\dot{\gamma}_{j}\partial_{x_{j}}\\
 & =\dot{\gamma}_{n+1}\partial_{u}+\sum_{j=1}^{n}\dot{\gamma}_{j}X_{j}\end{alignat*}
almost everywhere. Since the vector fields $U_{j}$, $j=1,\ldots,n$
are not linear independent, we may write $\dot{\gamma}_{n+1}\partial_{u}$
as a linear combination of $U_{1},\ldots,U_{n}$ in different ways.
So\[
\left\Vert \dot{\gamma}(t)\right\Vert _{\gamma(t)}^{2}=\sum_{j=1}^{n}\dot{\gamma}_{j}(t)^{2}+\inf\left\{ \sum_{j=1}^{n}\alpha_{j}^{2}:\;\dot{\gamma}_{n+1}(t)=\sum_{j=1}^{n}\alpha_{j}\gamma_{j}(t)\right\} .\]
To minimize this convex functional, one can use the method of Langrange
multiplier, which gives\[
\left\Vert \dot{\gamma}(t)\right\Vert _{\gamma(t)}^{2}=\frac{\dot{\gamma}_{n+1}(t)^{2}}{\sum_{j=1}^{n}\gamma_{j}(t)^{2}}+\sum_{j=1}^{n}\dot{\gamma}_{j}(t)^{2}.\]
Now let $\gamma=(\gamma_{(1)},\gamma_{(2)})$, with $\gamma_{(1)}=(\gamma_{1},\ldots,\gamma_{n})$,
$\gamma_{(2)}=\gamma_{n+1}$, then\[
\left\Vert \dot{\gamma}(t)\right\Vert _{\gamma(t)}^{2}=\frac{\dot{\gamma}_{(2)}(t)^{2}}{\left|\gamma_{(1)}(t)\right|^{2}}+\left|\dot{\gamma}_{(1)}(t)\right|^{2}.\]
To calculate the Carnot-Carathéodory distance, one can minimize the
{}``energy'' integral \begin{equation}
\int\limits _{0}^{1}\left\Vert \dot{\gamma}(t)\right\Vert _{\gamma(t)}^{2}dt=\int\limits _{0}^{1}\left(\frac{\dot{\gamma}_{(2)}(t)^{2}}{\left|\gamma_{(1)}(t)\right|^{2}}+\left|\dot{\gamma}_{(1)}(t)\right|^{2}\right)dt.\end{equation}
This leads to the Euler-Lagrange equations, which gives us local minimizing
curves, i.e. geodesics:\begin{subequations}\begin{eqnarray}
\ddot{\gamma}_{(1)}+\frac{\dot{\gamma}_{(2)}^{2}}{\left|\gamma_{(1)}\right|^{4}}\gamma_{(1)} & = & 0\\
\frac{d}{dt}\frac{\dot{\gamma}_{(2)}}{\left|\gamma_{(1)}\right|^{2}} & = & 0.\end{eqnarray}
\end{subequations} Solving these equations yields the following result
(see \cite{Greiner2002,Calin2005,Meyer2006}).

All geodesics $\gamma$ starting in $(x_{1},u_{1})\in\mathbb{R}^{n}\times\mathbb{R}$
are given by $\gamma^{b,c}=\left(\gamma_{(1)}^{b,c},\gamma_{(2)}^{b,c}\right):[0,1]\rightarrow\mathbb{R}^{n}\times\mathbb{R}$,
where\begin{subequations}\label{eq:geo1}\begin{eqnarray}
\gamma_{(1)}^{b,c}(t) & = & \frac{c}{b}\sin(bt)+x_{1}\cos(bt)\label{eq:__1}\\
\gamma_{(2)}^{b,c}(t) & = & \frac{\left|c\right|^{2}}{b}\left(\frac{t}{2}-\frac{\sin(2bt)}{4b}\right)+\frac{x_{1}\cdot c}{b}\sin^{2}(bt)\label{eq:__1}\\
 &  & +\left|x_{1}\right|^{2}b\left(\frac{t}{2}+\frac{\sin(2bt)}{4b}\right)+u_{1}\nonumber \end{eqnarray}
\end{subequations}with parameters $b\in\mathbb{R},\; b\neq0,\; c\in\mathbb{R}^{n}$,
and \begin{equation}
\gamma^{0,c}(t)=(ct+x_{1},u_{1})\label{eq:geo2}\end{equation}
which is the limiting case $b\rightarrow0$. 

The length is given by\begin{equation}
L(\gamma^{b,c})=\sqrt{\left|c\right|^{2}+\left|x_{1}\right|^{2}b^{2}}.\label{eq:geo1-length}\end{equation}
To calculate the Carnot-Carathéodory distance, we are interested in
the shortest geodesic joining two given points $(x_{1},u_{1}),(x,u)\in\mathbb{R}^{n}\times\mathbb{R}$. 

In our further analysis we will assume that $u_{1}=0$, since $\gamma$
is a geodesic connecting $(x_{1},0)$, $(x,u)$, if and only if $\gamma+(0,u_{1})$
is a geodesic connecting $(x_{1},u_{1})$, $(x_{1},u+u_{1})$; with
the same length. 

We will also assume that $u\geq0$, since $\gamma_{(1)}^{b,c}=\gamma_{(1)}^{-b,c}$
and $\gamma_{(2)}^{b,c}=-\gamma_{(2)}^{-b,c}$.

Furthermore observe, that given $b\in\mathbb{R}\setminus\pi\mathbb{Z}^{\ast}$
the parameter $c\in\mathbb{R}^{n}$ is uniquely determined by $c=\frac{b}{\sin b}\left(x-x_{1}\cos b\right)$,
due to the boundary conditions. We will denote this geodesic by $\gamma^{b}$
for short.

Given $b\in\pi\mathbb{Z}^{\ast}$ (note that this means $x=\pm x_{1}$),
the parameter $c\in\mathbb{R}^{n}$ is determined by the whole sphere
$\left|c\right|=\sqrt{2bu-\left|x_{1}\right|^{2}b^{2}}$. But any
$c\in\mathbb{R}^{n}$ with $\left|c\right|=\sqrt{2bu-\left|x_{1}\right|^{2}b^{2}}$
will give the same length.

\begin{thm}
Given $(x_{1},0),(x,0)\in\mathbb{R}^{n}\times\mathbb{R}$, the only
geodesic connecting them is given by\begin{equation}
\gamma^{0}(t)=\left(x_{1}+t\left(x-x_{1}\right),0\right)\end{equation}
and its length, which is equal to the Carnot-Carathéodory distance,
is\begin{equation}
L(\gamma^{0})=d_{CC}\left(\left(x_{1},0\right),\left(x,0\right)\right)=|x-x_{1}|.\end{equation}

\end{thm}
\begin{proof}
We have to prove that $b=0$ is the only parameter that allows connecting
$(x_{1},0)$ and $(x,0)$. But from the Euler-Lagrange equation, we
know that $\dot{\gamma}_{(2)}=b\left|\gamma_{(1)}\right|^{2}.$ That
means $\dot{\gamma}_{(2)}(t)\geq0$ or $\dot{\gamma}_{(2)}(t)\leq0$
for all $t\geq0$. But since we want $\gamma_{(2)}(0)=0$ and $\gamma_{(2)}(1)=0$,
we conclude that $\dot{\gamma}_{(2)}=0$, so that $b=0$, hence $c=x-x_{1}$.
\end{proof}
\begin{thm}
Given $u>0$, there are infinitely many geodesics $\gamma^{m\pi,c_{m}},\; m\in\mathbb{N}$,
connecting $(0,0)\in\mathbb{R}^{n}\times\mathbb{R}$ and $(0,u)\in\mathbb{R}^{n}\times\mathbb{R}$
given by \eqref{eq:geo1} with $c_{m}\in\mathbb{R}^{n},$ $\left|c_{m}\right|=\sqrt{2m\pi u}$.
Their lengths are\begin{equation}
L(\gamma^{m\pi,c_{m}})=\sqrt{2m\pi u}.\end{equation}
In particular the shortest geodesics in this case, which give the
Carnot-Carathéodory distance, are $\gamma^{\pi,c_{1}}$, $\left|c_{1}\right|=\sqrt{2\pi u}$;
so\begin{equation}
d_{CC}\left((0,0),(0,u)\right)=\sqrt{2\pi u}.\end{equation}

\end{thm}
\begin{proof}
The boundary conditions $\gamma_{(1)}^{b,c}(1)=0$ and $\gamma_{(2)}^{b,c}(1)=u$
are equivalent to\[
0=\frac{c}{b}\sin b\quad\mbox{and}\quad u=\frac{\left|c\right|^{2}}{2b}\left(1-\frac{\sin(2b)}{2b}\right).\]
Since $u\neq0$, we have $c\neq0$. But then there exists $m\in\mathbb{N}$
with $b=m\pi$, so that $\left|c\right|^{2}=2m\pi u$. This gives
the claim.
\end{proof}
\begin{figure}[h] 
\label{fig:geo00}
\centering
\mypsteps{\includegraphics{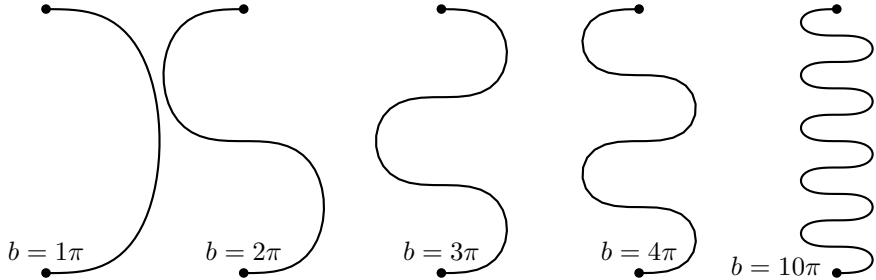}}{\psset{unit=1,algebraic=true}
\pspicture*(-.5,-.3)(11,4)
\def\opi{3.14159}
\def\geo#1#2{(2*#1/#2/\opi)^0.5*sin(#2*\opi*t)|#1*(t-sin(2*#2*\opi*t)/2/#2/\opi)}
\multido{\ib=1+1}{4}{%
\parametricplot{0}{1}{2.6*\ib-2.6+\geo{3.5}\ib}
\FPmul\nxx\ib{2.6}\FPsub\nx\nxx{2.6}
\psdots(\nx,0)(\nx,3.5)
\rput(\nx,.3){$b=\ib\pi$}}
\parametricplot[plotpoints=200]{0}{1}{10.4+\geo{3.5}{10}}
\psdots(10.4,0)(10.4,3.5)
\rput(9.6,.1){$b=10\pi$}
\endpspicture}
\caption{Geodesics starting at $(0,0)$ and ending in $(0,u)$ with different parameters.}
\end{figure}

Now we consider the case where we want to connect $(x_{1},0)$ and
$(x,u)$ with $x\neq\pm x_{1}$, $u>0$. The first lemma is just a
rewritten formulation of the boundary conditions $\gamma^{b,c}(0)=(x_{1},0)\;\mbox{and}\;\gamma^{b,c}(1)=(x,u)$.

\begin{lem}
\label{lem:allg geo gln}Suppose $(x_{1},0),(x,u)\in\mathbb{R}^{n}\times\mathbb{R}$,
$x\neq\pm x_{1},$ $u>0$. Then $\gamma^{b,c}$ given by \eqref{eq:geo1}
is a geodesic connecting these two points, iff \[
c=\frac{b}{\sin b}x-x_{1}b\cot b\]
 and $b$ is any solution of\begin{equation}
\frac{2u}{\left|x_{1}\right|^{2}+\left|x\right|^{2}}=\frac{b}{\sin^{2}b}-\cot b+\frac{2x\cdot x_{1}}{\left|x_{1}\right|^{2}+\left|x\right|^{2}}\frac{1-b\cot b}{\sin b}.\label{eq:mu}\end{equation}
The square of the length of $\gamma^{b}=\gamma^{b,c}$ is then given
by\begin{subequations}\label{eq:length}\begin{eqnarray}
L^{2}(\gamma^{b,c}) & = & \frac{b^{2}}{\sin^{2}b}\left(\left|x_{1}\right|^{2}+\left|x\right|^{2}-2x\cdot x_{1}\cos b\right)\label{eq:__1}\\
 & = & 2bu+\left(\left|x_{1}\right|^{2}+\left|x\right|^{2}\right)b\cot b-2x\cdot x_{1}\frac{b}{\sin b}.\label{eq:__1}\end{eqnarray}
\end{subequations}
\end{lem}

\begin{proof}
To find those geodesics connecting two points $(x_{1},0)$ and $(x,u)$,
we have to find $b\neq0$ and $c\in\mathbb{R}^{n}$, such that $\gamma_{(1)}^{b,c}(1)=x$
and $\gamma_{(2)}^{b,c}(1)=u$. Since $b\in\pi\mathbb{Z}$ means $x=\pm x_{1}$,
we have that $b\not\in\pi\mathbb{Z}$. Now $\gamma_{(2)}^{b,c}(1)=u$
gives:\begin{alignat}{1}
2u & =\frac{\left|c\right|^{2}}{b}\left(1-\frac{\sin(2b)}{2b}\right)+2\frac{x_{1}\cdot c}{b}\sin^{2}(b)+\left|x_{1}\right|^{2}b\left(1+\frac{\sin(2b)}{2b}\right)\nonumber \\
 & =\frac{1}{b}\left(\left|c\right|^{2}+\left|x_{1}\right|^{2}b^{2}\right)-\frac{\sin(2b)}{2b^{2}}\left(\left|c\right|^{2}-\left|x_{1}\right|^{2}b^{2}\right)+2\frac{x_{1}\cdot c}{b}\sin^{2}b\label{eq:lem1e0}\end{alignat}
And $\gamma_{(1)}^{b,c}(1)=x$ gives\begin{equation}
\frac{b}{\sin b}x=c+x_{1}b\cot b,\label{eq:lem1e1}\end{equation}
so that\begin{eqnarray}
\frac{b^{2}}{\sin^{2}b}\left|x\right|^{2} & = & \left|c\right|^{2}+2c\cdot x_{1}b\cot b+\left|x_{1}\right|^{2}b^{2}\cot^{2}b\nonumber \\
 & = & \left|c\right|^{2}+2c\cdot x_{1}b\cot b-\left|x_{1}\right|^{2}b^{2}+\frac{\left|x_{1}\right|^{2}b^{2}}{\sin^{2}b}\label{eq:lem1e2}\end{eqnarray}
and\begin{alignat}{1}
\left|c\right|^{2} & =\frac{b^{2}}{\sin^{2}b}\left|x\right|^{2}-2x_{1}\cdot x\frac{b^{2}\cot b}{\sin b}+\left|x_{1}\right|^{2}b^{2}\cot^{2}b\nonumber \\
 & =\frac{b^{2}}{\sin^{2}b}\left(\left|x\right|^{2}+\left|x_{1}\right|^{2}\right)-2x_{1}\cdot x\frac{b^{2}\cot b}{\sin b}-\left|x_{1}\right|^{2}b^{2}\nonumber \\
 & =\frac{b^{2}}{\sin^{2}b}\left(\left|x\right|^{2}+\left|x_{1}\right|^{2}-2x\cdot x_{1}\cos b\right)-\left|x_{1}\right|^{2}b^{2}.\label{eq:lem1e3}\end{alignat}
We can continue with \eqref{eq:lem1e0} by using equation \eqref{eq:lem1e2}
\begin{alignat*}{1}
2u & =\frac{1}{b}\left(\left|c\right|^{2}+\left|x_{1}\right|^{2}b^{2}\right)-\cot b\left(\left|x\right|^{2}-\left|x_{1}\right|^{2}\right)+2\frac{c\cdot x_{1}}{b}\cos^{2}b+2\frac{c\cdot x_{1}}{b}\sin^{2}b\\
 & =\frac{1}{b}\left(\left|c\right|^{2}+\left|x_{1}\right|^{2}b^{2}\right)-\cot b\left(\left|x\right|^{2}-\left|x_{1}\right|^{2}\right)+2\frac{c\cdot x_{1}}{b}\\
 & =\frac{1}{b}\left(\left|c\right|^{2}+\left|x_{1}\right|^{2}b^{2}\right)-\cot b\left(\left|x\right|^{2}+\left|x_{1}\right|^{2}\right)+2\frac{c\cdot x_{1}+\left|x_{1}\right|^{2}b\cot b}{b}.\end{alignat*}
Now use equation \eqref{eq:lem1e1} and \eqref{eq:lem1e3} \begin{alignat*}{1}
2u & =\frac{1}{b}\left(\left|c\right|^{2}+\left|x_{1}\right|^{2}b^{2}\right)-\cot b\left(\left|x\right|^{2}+\left|x_{1}\right|^{2}\right)+2\frac{x\cdot x_{1}}{\sin b}\\
 & =\left(\frac{b}{\sin^{2}b}-\cot b\right)\left(\left|x\right|^{2}+\left|x_{1}\right|^{2}\right)+2x\cdot x_{1}\frac{1-b\cot b}{\sin b}.\end{alignat*}
And the square of the length is given by\begin{alignat*}{1}
L^{2}(\gamma^{b,c}) & =\frac{b^{2}}{\sin^{2}b}\left(\left|x\right|^{2}+\left|x_{1}\right|^{2}-2x\cdot x_{1}\cos b\right)\\
 & =2bu+b\cot b\left(\left|x\right|^{2}+\left|x_{1}\right|^{2}\right)-2\frac{b}{\sin b}x\cdot x_{1}.\end{alignat*}

\end{proof}
\begin{defn*}
For $b\in\mathbb{R}\setminus\pi\mathbb{Z}^{\ast},$ $-1\leq a\leq1$
define:\begin{eqnarray}
\mu(b) & := & \mu(b,a):=\frac{b}{\sin^{2}b}-\cot b+a\frac{1-b\cot b}{\sin b}\\
l(b) & := & l(b,a):=\frac{b^{2}}{\sin^{2}b}\left(1-a\cos b\right).\end{eqnarray}

\end{defn*}
In the following we will often use the abbreviation \[
R=\sqrt{\left|x\right|^{2}+\left|x_{1}\right|^{2}}\mbox{\;\mbox{and}\;}a=\frac{2x_{1}\cdot x}{\left|x_{1}\right|^{2}+\left|x\right|^{2}}.\]

Next we will study the functions $\mu$ and $\tilde{\mu},\,\hat{\mu}$,
which will be introduced in the next lemma. (Note that $\hat{\mu}$
in this section does not mean Fourier transform!)

\begin{lem}
\label{lem:convex comb mu}For $b\in\mathbb{R}\setminus\pi\mathbb{Z}^{\ast}$,
$0\leq a\leq1$, $\mu$ is a convex combination, i.e.\begin{equation}
\mu(b)=(1-a)\tilde{\mu}+a\hat{\mu}\left(\frac{b}{2}\right),\end{equation}
and for $-1\leq a<0$:\begin{equation}
\mu(b)=(1+a)\tilde{\mu}(b)-a\tilde{\mu}\left(\frac{b}{2}\right),\end{equation}
where\begin{eqnarray}
\tilde{\mu}(b) & := & \frac{b}{\sin^{2}b}-\cot b,\; b\in\mathbb{R}\setminus\pi\mathbb{Z}^{\ast},\\
\hat{\mu}(b) & := & \frac{b}{\cos^{2}b}+\tan b,\; b\in\mathbb{R}\setminus\left\{ \left(k+\frac{1}{2}\right)\pi:\; k\in\mathbb{Z}\right\} .\end{eqnarray}

\end{lem}
\begin{proof}
For $b\in\mathbb{R}\setminus\pi\mathbb{Z}^{\ast}$ :\begin{align*}
\mu(b) & =(1-a)\tilde{\mu}(b)+a\frac{1-b\cot b+\frac{b}{\sin b}-\cos b}{\sin b}\\
 & =(1-a)\tilde{\mu}(b)+2a\frac{\sin^{2}\frac{b}{2}\left(1+\frac{b}{\sin b}\right)}{\sin b}\\
 & =(1-a)\tilde{\mu}(b)+a\left(\frac{\frac{b}{2}}{\cos^{2}\frac{b}{2}}+\tan\frac{b}{2}\right)\\
 & =(1-a)\tilde{\mu}(b)+a\hat{\mu}\left(\frac{b}{2}\right)\end{align*}
and\begin{align*}
\mu(b) & =(1+a)\tilde{\mu}(b)-a\frac{b\cot b-1+\frac{b}{\sin b}-\cos b}{\sin b}\\
 & =(1+a)\tilde{\mu}(b)-2a\frac{\cos^{2}\frac{b}{2}\left(\frac{b}{\sin b}-1\right)}{\sin b}\\
 & =(1+a)\tilde{\mu}(b)-a\left(\frac{\frac{b}{2}}{\sin^{2}\frac{b}{2}}-\cot\frac{b}{2}\right)\\
 & =(1+a)\tilde{\mu}(b)-a\tilde{\mu}\left(\frac{b}{2}\right)\end{align*}

\end{proof}
\begin{rem*}
Observe that all these functions $\mu,\tilde{\mu},\hat{\mu}$ are
odd. The following lemmata will show that $\tilde{\mu}|_{[0,\infty)}\geq0,\;\hat{\mu}|_{[0,\infty)}\geq0$
and therefore $\mu|_{[0,\infty)}\geq0$. Since we assumed that $u\geq0$,
we only need information of these functions for $b\geq0$.

\begin{figure}[h]
\centering
\subfigure[{Plotting of $\tilde\mu(b)$ (gray), $\hat\mu(b/2)$ (gray) and $\mu(b)$ (black) with parameter $a=0.7$.}]{
\mypsteps{\includegraphics{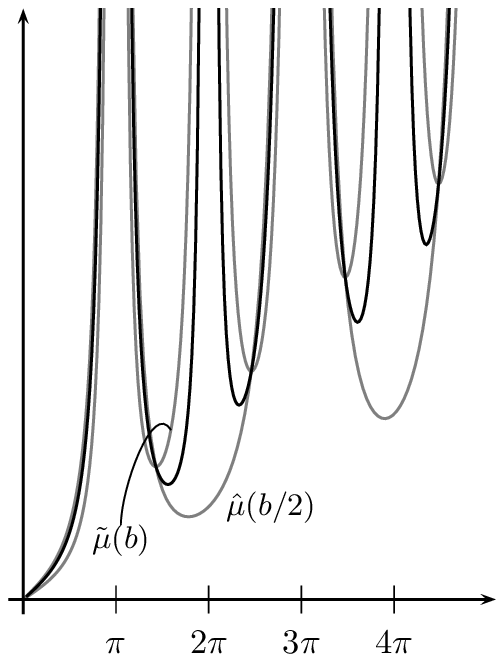}}{}


}
\hspace{0.5cm}%
\subfigure[{Plotting of $\tilde\mu(b)$ (gray), $\tilde\mu(b/2)$ (gray) and $\mu(b)$ (black) with parameter $a=-0.7$.}]{
\mypsteps{\includegraphics{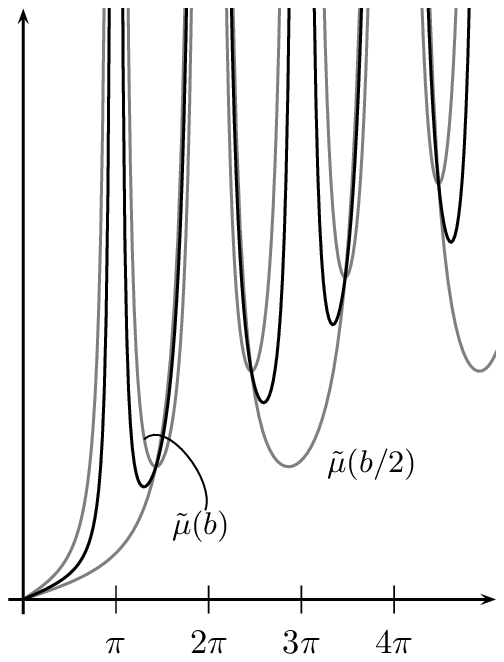}}{%
\psset{xunit=.3,yunit=0.3,algebraic=true,trigLabels=true}
\pspicture*(-.5,-2)(16,20)
\multido{\rsi=0.0+3.1416}{5}{
\FPadd\nsil\rsi{0.05}
\FPadd\nsir\rsi{3.12}
\psplot[linewidth=.8pt,linecolor=gray]{\nsil}{\nsir}{x/(sin(x)^2)-cos(x)/sin(x)}}
\multido{\rsi=0.0+6.2832}{3}{
\FPadd\nsil\rsi{0.05}
\FPadd\nsir\rsi{6.25}
\psplot[linewidth=.8pt,linecolor=gray]{\nsil}{\nsir}{x/(2*sin(x/2)^2)-cos(x/2)/sin(x/2)}}
\multido{\rsi=0.0+3.1416}{5}{
\FPadd\nsil\rsi{0.05}
\FPadd\nsir\rsi{3.12}
\psplot[linewidth=.8pt]{\nsil}{\nsir}{(x/(sin(x)^2)-cos(x)/sin(x))-0.7*(1-x*cos(x)/sin(x))/sin(x)}}
\uput[r](9.7,4.6){\small{$\tilde\mu(b/2)$}}
\rput(6,2.5){\Rnode{A}{\small{$\tilde\mu(b)$}}}
\pnode(4.1,5.421){G}
\nccurve[angleA=70,angleB=45,linewidth=.5pt]{A}{G}
\psaxes[labels=x,ticks=x,Dx=1,dx=3.14]{->}(0,0)(-0.5,-0.5)(16,20)
\endpspicture
}
}
\end{figure}

\end{rem*}
The function $\tilde{\mu}$, which also appears in the study of geodesics
on the Heisenberg group, where intensively studied by Beals, Gaveau,
Greiner in \cite{Beals2000}. We will use the following result:

\begin{lem}
\label{lem:tilde mu}The function $\tilde{\mu}$ is a monotone increasing
diffeomorphism of the interval $[0,\pi)$ onto $[0,\infty)$. On each
interval $\left(m\pi,(m+1)\pi\right),\, m\in\mathbb{N}$, $\tilde{\mu}$
has a unique critical point $\tilde{b}_{m}$, which is a minimum.
$\tilde{b}_{m}$ is implicitly given as the solution of the equation
$1-b\cot b=0$ in the corresponding interval. On this interval, $\tilde{\mu}$
decreases strictly from $+\infty$ to $\tilde{\mu}(\tilde{b}_{m})$
and then increases strictly to $+\infty$. Moreover,\begin{equation}
\tilde{\mu}(\tilde{b}_{m})+\pi<\tilde{\mu}(\tilde{b}_{m+1}),\end{equation}
and \begin{equation}
\tilde{\mu}(\tilde{b}_{m})>m\pi.\label{eq: mu tilde geq m pi}\end{equation}

\end{lem}
The function $\hat{\mu}$ has been studied by Calin, Chang, Greiner,
Kannai in \cite{Calin2005}:

\begin{lem}
\label{lem:hat mu}The function $\hat{\mu}$ is a monotone increasing
diffeomorphism of $[0,\frac{\pi}{2})$ onto $[0,\infty)$. On each
interval $\left(m\pi+\frac{\pi}{2},(m+1)\pi+\frac{\pi}{2}\right),\; m\in\mathbb{N}_{0}$,
$\hat{\mu}$ has an unique critical point $\hat{b}_{m}$, which is
a minimum. $\hat{b}_{m}$ is implicitly given as the solution of the
equation $1+b\tan b=0$ in the corresponding interval. On this interval,
$\hat{\mu}$ decreases strictly from $+\infty$ to $\hat{\mu}(\hat{b}_{m})$
and then increases strictly to $+\infty$. Moreover,\begin{equation}
\hat{\mu}(\hat{b}_{m})\geq\pi\left(m+\frac{1}{2}\right).\label{eq:hat mu geq m pi}\end{equation}

\end{lem}
\begin{rem*}
Moreover $\tilde{\mu}''(b)>0,\; b\not\in\pi\mathbb{N}$ and $\hat{\mu}''(b)>0,\; b-\frac{\pi}{2}\not\in\mathbb{N}_{0}$,
i.e. $\tilde{\mu}$ is strictly convex on each interval $\left(m\pi,(m+1)\pi\right),\, m\in\mathbb{N}$
and $\hat{\mu}$ is strictly convex on each interval $\left(m\pi+\frac{\pi}{2},(m+1)\pi+\frac{\pi}{2}\right),\, m\in\mathbb{N}_{0}$.
This is another result of \cite{Beals2000,Calin2005}. Hence $\mu$,
as a convex combination, is also strictly convex on each interval
$\left(m\pi,(m+1)\pi\right),\, m\in\mathbb{N}$.
\end{rem*}
Combining the previous two lemmata we get the following result for
the function $\mu$:

\begin{lem}
Let $-1<a<1$. The function $\mu$ is monotone increasing on $[0,\pi)$
onto $[0,\infty)$. On each interval $(m\pi,(m+1)\pi),\; m\in\mathbb{N}$,
$\mu$ has an unique critical point $b_{m}$, which is a minimum.
For $a\geq0$ and $m\in\mathbb{N}_{0}$ \begin{eqnarray}
(2m+1)\pi & <\tilde{b}_{2m+1}\leq b_{2m+1}\leq2\hat{b}_{m}< & (2m+2)\pi,\label{eq:b-verteilung-ungerade}\\
(2m+2)\pi & <b_{2m+2}\leq\tilde{b}_{2m+2}< & (2m+3)\pi;\nonumber \end{eqnarray}
and for $a<0$ and $m\in\mathbb{N}$ \begin{eqnarray}
2m\pi & <\tilde{b}_{2m}\leq b_{2m}\leq2\tilde{b}_{m}< & (2m+1)\pi,\label{eq:b-verteilung-gerade}\\
(2m+1)\pi & <b_{2m+1}\leq\tilde{b}_{2m+1}< & (2m+2)\pi.\nonumber \end{eqnarray}
On each interval $(m\pi,(m+1)\pi),\; m\in\mathbb{N}$, $\mu$ decreases
strictly from $+\infty$ to $\mu(b_{m})$ and then increases strictly
to $+\infty$. Moreover for all $m\in\mathbb{N}$,\begin{eqnarray}
\mu(b_{m}) & \geq & \frac{m-1}{2}\pi,\label{eq:mu geq m 2}\\
l(b_{m})-b_{m}\mu(b_{m}) & = & 1-a\delta(b_{m}),\label{eq:delta-equ}\end{eqnarray}
where\begin{equation}
\delta(b):=\cos b+\frac{b}{2}\sin b.\end{equation}

\end{lem}
\begin{proof}
It is easy to see that $m\pi<\tilde{b}_{m}<m\pi+\frac{\pi}{2}<\hat{b}_{m}<(m+1)\pi$,
see figure \ref{fig:tancot}.

\begin{figure}[h]
\label{fig:tancot}
\centering
\mypsteps{\includegraphics{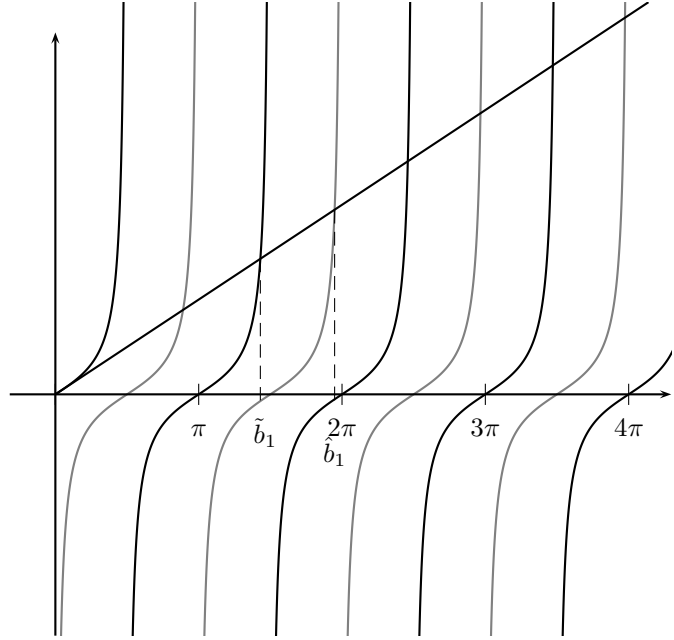}}{
%
}
\caption{The graphs of the functions $\tan$ (black) and $-\cot$ (gray) and the line $x\mapsto x$.}
\end{figure}

Since \begin{eqnarray*}
1-b\cot b & = & 1+\frac{b}{2}\tan\frac{b}{2}-\frac{b}{2}\cot\frac{b}{2}\\
 & \geq & \begin{cases}
1+\frac{b}{2}\tan\frac{b}{2} & b\in\left((2m+1)\pi,(2m+2)\pi\right)\\
1-\frac{b}{2}\cot\frac{b}{2} & b\in\left(2m\pi,(2m+1)\pi\right)\end{cases}\end{eqnarray*}
and $\tilde{b}_{2m+1},2\hat{b}_{m}\in\left((2m+1)\pi,(2m+2)\pi\right)$
it follows that $\tilde{b}_{2m+1}<2\hat{b}_{m}$; also, since $\tilde{b}_{2m},2\tilde{b}_{m}\in\left(2m\pi,(2m+1)\pi\right)$,
it follows that $\tilde{b}_{2m}<2\tilde{b}_{m}$. 

\begin{figure}
\centering
\mypsteps{\includegraphics{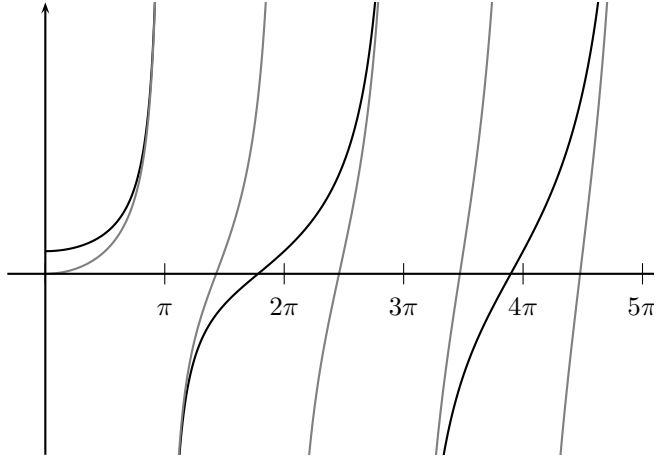}}{
\psset{xunit=.5,yunit=.3}
\psset{trigLabels=true,algebraic=true}
\pspicture*(-1,-8)(16,12)
\psplot[plotpoints=300]{0}{3.141}{1+x/2*sin(x/2)/cos(x/2) }
\multido{\rsi=3.1416+6.2832}{3}{
\FPadd\nsil\rsi{0.005}
\FPadd\nsir\rsi{6.27} \psplot[plotpoints=300]{\nsil}{\nsir}{1+x/2*sin(x/2)/cos(x/2) }
}
\multido{\rsi=0.0+3.1416}{5}{
\FPadd\nsil\rsi{0.005}
\FPadd\nsir\rsi{3.13}
\psplot[plotpoints=300,linecolor=gray]{\nsil}{\nsir}{1-x*cos(x)/sin(x)}
}
\psaxes[labels=x,ticks=x,Dx=1,dx=3.1416]{->}(0,0)(-1,-8)(17,12)
\endpspicture
}
\caption{The graphs of the functions $x\mapsto 1+\frac x 2\tan\frac x 2$ (black) and $x\mapsto 1-x\cot x$ (gray).}
\end{figure}

If $a\geq0$, $\mu'$ is a convex combination of $\tilde{\mu}'$ and
$\frac{1}{2}\hat{\mu}'\left(\frac{b}{2}\right)$. On the interval
$\left((2m+1)\pi,(2m+2)\pi\right)$ both functions are monotone increasing
with zero $\tilde{b}_{2m+1}$ of $\tilde{\mu}'$ lying left of $2\hat{b}_{m}$,
so we have \[
\mu'\left(\tilde{b}_{2m+1}\right)=\frac{a}{2}\hat{\mu}'\left(\frac{\tilde{b}_{2m+1}}{2}\right)\leq0\mbox{ and }\mu'\left(2\hat{b}_{m}\right)=(1-a)\tilde{\mu}'\left(2\hat{b}_{m}\right)>0.\]
Hence $\mu'$ has a zero $b_{m}$ in between: $\tilde{b}_{2m+1}\leq b_{m}<2\hat{b}_{m}$.
Due to the fact that $\mu$ is strictly convex, this critical point
is unique in the interval $((2m+1)\pi,(2m+2)\pi)$, and also a minimum.

In the interval $\left((2m+2)\pi,(2m+3)\pi\right)$ still both functions
are monotone increasing, but only $\tilde{\mu}'$ has a zero at $\tilde{b}_{2m+2}$.
$\frac{1}{2}\hat{\mu}'\left(\frac{b}{2}\right)$ is positive on this
interval and since $\tilde{\mu}'(b)$ tends to $-\infty$ as $\begin{CD}b@>{b>(2m+2)\pi}>>(2m+2)\pi\end{CD}$,
there is a zero $b_{2m+2}$ of $\mu'$ in this interval $((2m+2)\pi,(2m+3)\pi)$
with $(2m+2)\pi\leq b_{2m+2}<\tilde{b}_{2m+2}$. This critical point
is unique and a minimum, since $\mu$ is convex.

The case $a<0$ is similar, if we use $\frac{1}{2}\tilde{\mu}\left(\frac{b}{2}\right)$
instead of $\frac{1}{2}\hat{\mu}\left(\frac{b}{2}\right)$ and the
fact that in the interval $\left(2m\pi,(2m+1)\pi\right)$: $\tilde{b}_{2m}<2\tilde{b}_{m}$
are zeros of these two functions and $\frac{1}{2}\tilde{\mu}\left(\frac{b}{2}\right)$
has no zero in $\left((2m+1)\pi,(2m+2)\pi\right)$. 

To obtain the lower bound \eqref{eq:mu geq m 2} on $\mu$, we consider
the following cases:
\begin{lyxlist}{00}
\item [{\textbf{1.~case~$a\geq0$~and~$m$~even:}}] ~\\
Since $\frac{b_{m}}{2}\in\left(\left(\frac{m-2}{2}+\frac{1}{2}\right)\pi,\left(\frac{m}{2}+\frac{1}{2}\right)\pi\right)$,
we have\begin{alignat*}{1}
\mu(b_{m}) & =(1-a)\tilde{\mu}(b_{m})+a\hat{\mu}\left(\frac{b_{m}}{2}\right)\\
 & \geq(1-a)m\pi+a\frac{m-1}{2}\pi\\
 & \geq\frac{m-1}{2}\pi.\end{alignat*}

\item [{\textbf{2.~case~$a\geq0$~and~$m$~odd:}}] \textbf{~}\\
Since $\frac{b_{m}}{2}\in\left(\left(\frac{m-1}{2}+\frac{1}{2}\right)\pi,\frac{m+1}{2}\pi\right)$,
we have\begin{alignat*}{1}
\mu(b_{m}) & \geq(1-a)m\pi+a\frac{m}{2}\pi\\
 & \geq\frac{m}{2}\pi>\frac{m-1}{2}\pi.\end{alignat*}

\item [{\textbf{3.~case~$a<0$~and~$m$~even:}}] ~\\
Since $\frac{b_{m}}{2}\in\left(\frac{m}{2}\pi,\frac{m+1}{2}\pi\right)$,
we have\begin{alignat*}{1}
\mu(b_{m}) & =(1+a)\tilde{\mu}(b_{m})-a\tilde{\mu}\left(\frac{b_{m}}{2}\right)\\
 & \geq(1+a)m\pi-a\frac{m}{2}\pi\\
 & \geq\frac{m}{2}\pi>\frac{m-1}{2}\pi.\end{alignat*}

\item [{\textbf{4.~case~$a<0$~and~$m$~odd:}}] ~\\
Since $\frac{b_{m}}{2}\in\left(\frac{m-1}{2}\pi,\frac{m+1}{2}\pi\right)$,
we have\begin{alignat*}{1}
\mu(b_{m}) & \geq(1+a)m\pi-a\frac{m-1}{2}\pi\\
 & \geq\frac{m-1}{2}\pi.\end{alignat*}

\end{lyxlist}
At last, to establish the equation \eqref{eq:delta-equ}, we need
some more information of $b_{m}$, i.e. the critical points of $\mu$.
First observe, that for any $b\in\mathbb{R}\setminus\pi\mathbb{Z}^{\ast}$:\begin{equation}
b\mu(b)+b\cot b-\frac{b}{\sin b}a=l(b)\label{eq:l-b mu}\end{equation}
and\begin{alignat*}{1}
\frac{d}{db}\left(b\cot b-\frac{b}{\sin b}a\right) & =\cot b-\frac{b}{\sin^{2}b}-a\frac{1-b\cot b}{\sin b}\\
 & =-\mu(b).\end{alignat*}
So\begin{eqnarray}
l'(b) & = & \mu(b)+b\mu'(b)-\mu(b)\nonumber \\
 & = & b\mu'(b).\label{eq:l = b mu}\end{eqnarray}
To find the critical points of $\mu$, we differentiate:\begin{align*}
\mu'(b) & =2\frac{1-b\cot b}{\sin^{2}b}+a\frac{\frac{b}{\sin b}-\cos b-\cos b(1-b\cot b)}{\sin^{2}b}\\
 & =2\frac{1-b\cot b}{\sin^{2}b}+a\frac{b\sin b-2\cos b(1-b\cot b)}{\sin^{2}b}.\end{align*}
So\begin{equation}
\mu'(b)=0\Leftrightarrow a=\frac{1-b\cot b}{\cos b(1-b\cot b)-\frac{b}{2}\sin b}.\end{equation}
From this and \eqref{eq:l-b mu} we get: \begin{align*}
l(b_{m})-b_{m}\mu(b_{m}) & =b_{m}\cot b_{m}-\frac{b_{m}}{\sin b_{m}}a\\
 & =1-\left(1-b_{m}\cot b_{m}\right)-\frac{b_{m}}{\sin b_{m}}a\\
 & =1-a\left(\cos b_{m}\left(1-b_{m}\cot b_{m}\right)-\frac{b_{m}}{2}\sin b_{m}+\frac{b_{m}}{\sin b_{m}}\right)\\
 & =1-a\left(\cos b_{m}\left(1-b_{m}\cot b_{m}\right)+\frac{b_{m}}{2}\sin b_{m}+b_{m}\cos b_{m}\cot b_{m}\right)\\
 & =1-a\left(\cos b_{m}+\frac{b_{m}}{2}\sin b_{m}\right)\\
 & =1-a\delta(b_{m}).\end{align*}

\end{proof}
%
{}For the proof of the next theorem we need another lemma:

\begin{lem}
On each interval $\left(m\pi,(m+1)\pi\right),\; m\in\mathbb{N}$ $\delta$
has a unique zero, and a unique critical point at $\tilde{b}_{m}$,
which is a maximum with $\delta(\tilde{b}_{m})>0$, if m is even and
a minimum with $\delta(\tilde{b}_{m})<0$, if m is odd .
\end{lem}
\begin{figure}[h]
\centering
\mypsteps{\includegraphics{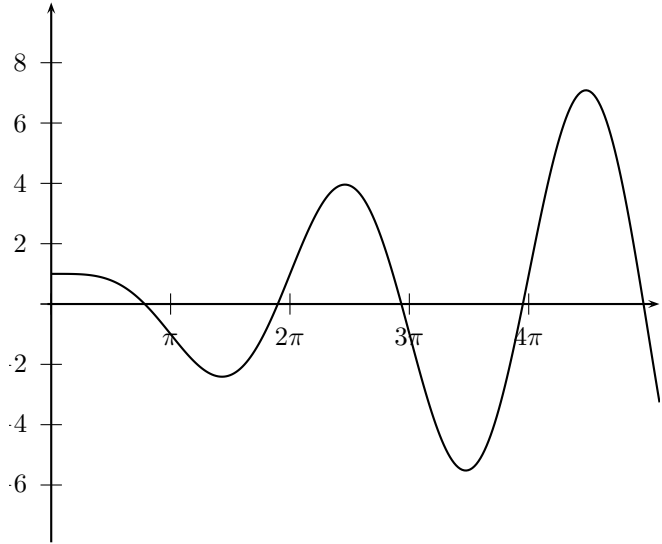}}{
\psset{xunit=.5,yunit=0.4,algebraic=true,trigLabels=true}
\pspicture*(-1.1,-8)(16,10)
  \psplot[plotpoints=300]{0.001}{16}{cos(x)+x*sin(x)/2}
\psaxes[ticks=all,dy=2,Dy=2,Dx=1,dx=3.14]{->}(0,0)(-0.1,-7.9)(16,10)
\endpspicture
}
\caption{The function $\delta$.}
\end{figure}

\begin{proof}
Since\begin{align*}
\delta(b)=0 & \Leftrightarrow\cos b+\frac{b}{2}\sin b=0\\
 & \Leftrightarrow2+b\tan b=0,\end{align*}
there is a unique zero in each interval $\left(m\pi,(m+1)\pi\right),\; m\in\mathbb{N}_{0}$.

The critical points of $\delta$ are given by\begin{align*}
\delta'(b)=0 & \Leftrightarrow-\sin b+\frac{1}{2}\sin b+\frac{b}{2}\cos b=0\\
 & \Leftrightarrow b\cot b-1=0\\
 & \Leftrightarrow b=\tilde{b}_{m}.\end{align*}
Since \begin{equation}
\delta''(b)=-\frac{b}{2}\sin b,\end{equation}
we see that $\tilde{b}_{m}$ is a maximum, if $m$ is even and a minimum,
if $m$ is odd.
\end{proof}
\begin{thm}
\label{thm:allg geo}Given $(x_{1},0),\,(x,u)\in\mathbb{R}^{n}\times\mathbb{R}$
with $x_{1}\neq\pm x$ and $u>0$, there are finitely many geodesics
joining these two points. These geodesics are given by \eqref{eq:geo1},
where $b$ is a solution of\begin{equation}
\frac{2u}{\left|x_{1}\right|^{2}+\left|x\right|^{2}}=\mu(b),\label{eq:solution}\end{equation}
and their lengths are strictly increasing with $b$. Moreover the
shortest geodesic joining $(x_{1},0)$ and $(x,u)$ is given by the
unique solution $b\in(0,\pi)$ of \eqref{eq:solution} in the interval
$(0,\pi)$. With this solution $b$ the Carnot-Carathéodory distance
is\begin{alignat}{1}
d_{CC}\left((x_{1},0),(x,u)\right) & =L(\gamma^{b})\label{eq:geo2-length}\\
 & =\frac{b}{\sin b}\sqrt{\left|x_{1}\right|^{2}+\left|x\right|^{2}-2x_{1}\cdot x\cos b}\nonumber \\
 & =\sqrt{2bu+\left(\left|x_{1}\right|^{2}+\left|x\right|^{2}\right)b\cot b-2\frac{b}{\sin b}x_{1}\cdot x}.\nonumber \end{alignat}

\end{thm}
\begin{figure}[ht]
\centering
\mypsteps{\includegraphics{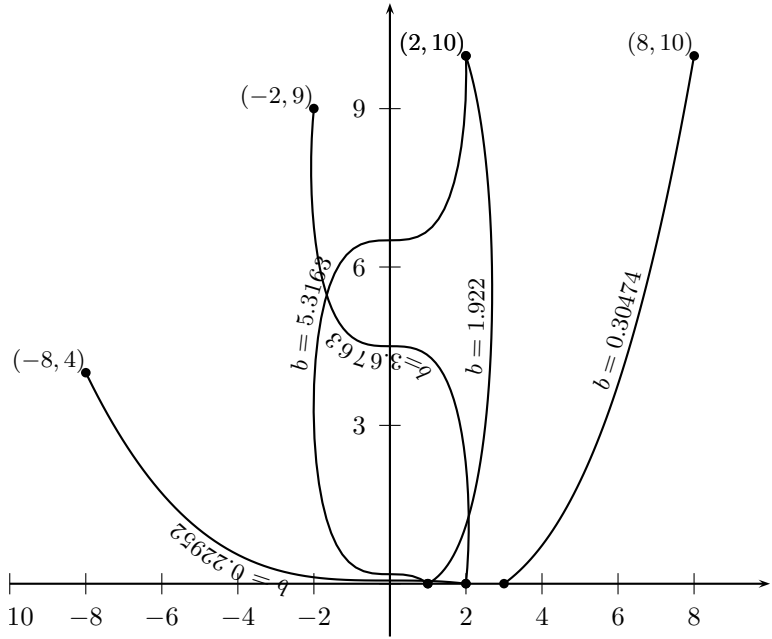}}{
\psset{yunit=.7,xunit=.5,VarStep=true,algebraic=true}
\pspicture*(-10,-1)(10,11)
\def\geo#1#2#3{(#2*sin(#1*t)/#1+#3*cos(#1*t))|(#2^2*(t/2-sin(2*#1*t)/4/#1)/#1+#3*#2*sin(#1*t)^2/#1+#3^2*#1*(t/2+sin(2*#1*t)/4/#1))} 
\def\geoplot#1#2#3#4{
\psdots(#1,0)(#2,#3)
\rput[rb](#2,#3){\small{$(#2,#3)$}}
\pstextpath[c](0,.1){
\parametricplot{0}{1}{\geo{#4}{(#2*#4/sin(#4)-#1*#4*cos(#4)/sin(#4))}{#1}}}
{\small{$b=#4$}}}
\geoplot{1} 2 {10} {1.922}
\geoplot{1} 2 {10} {5.3163}
\geoplot{3} 8 {10} {0.30474}
\geoplot{2}{-8}{4}{0.22952}
\geoplot{2}{-2}{9}{3.6763}
\psaxes[ticks=all,dx=2,Dx=2,dy=3,Dy=3]{->}(0,0)(-10,-1)(10,11)
\endpspicture
}
\caption{Geodesics joing different points.}
\end{figure}

\begin{proof}
From Lemma \ref{lem:allg geo gln} we have that the geodesics in this
case are implicitly given by solutions of the equation $\frac{2u}{\left|x_{1}\right|^{2}+\left|x\right|^{2}}=\mu(b)$.
The properties of the function $\mu$ show us that there is exactly
one solution in the interval $(0,\pi)$ and at most two solutions
in each interval $(m\pi,(m+1)\pi),\; m\in\mathbb{N}$. So it remains
to prove, that the values of $l$ at these solutions are strictly
increasing, i.e. if $s<t$ are two solutions, then $l(s)<l(t)$.

Let $s\in\left(m\pi,(m+1)\pi\right),\; m\in\mathbb{N}$, then we have
by \eqref{eq:l =00003D b mu}\begin{align*}
l(s) & =\int_{b_{m}}^{s}l'(b)db+l(b_{m})\\
 & =\int_{b_{m}}^{s}b\mu'(b)db+l(b_{m})\\
 & =b\mu(b)|_{b_{m}}^{s}-\int_{b_{m}}^{s}\mu(b)db+l(b_{m})\\
 & =s\mu(s)-\int_{b_{m}}^{s}\mu(b)db+l(b_{m})-b_{m}\mu(b_{m})\\
 & =s\mu(s)-\int_{b_{m}}^{s}\mu(b)db+1-a\delta(b_{m}).\end{align*}
So, if we have two solutions $s\leq t$ of \eqref{eq:solution} in
the same interval $\left(m\pi,(m+1)\pi\right),\; m\in\mathbb{N}$,
we see that\begin{align*}
l(s)-l(t) & =s\mu(s)-\int_{b_{m}}^{s}\mu(b)db-t\mu(t)+\int_{b_{m}}^{t}\mu(b)db\\
 & =(s-t)\mu(s)+\int_{s}^{t}\mu(b)db\\
 & \leq0.\end{align*}
Now let $s_{m},t_{m}\in\left(m\pi,(m+1)\pi\right),\; s_{m}\leq t_{m},\; m\in\mathbb{N}$
be the solutions of \eqref{eq:solution}. To compare the corresponding
{}``lengths'' $l$ at $t_{m}$ and $s_{m+1}$, notice that\begin{multline}
l(t_{m})-l(s_{m+1})\label{eq:l(t_m)-l(s_m+1)}\\
=-\left(s_{m+1}-t_{m}\right)\mu(t_{m})-\int_{b_{m}}^{t_{m}}\mu(b)db-\int_{s_{m+1}}^{b_{m+1}}\mu(b)db-\left(\delta(b_{m})-\delta(b_{m+1})\right)a.\end{multline}
Now consider the following cases:
\begin{lyxlist}{00}
\item [{\textbf{1.~case~$a\geq0$~and~$m$~even:}}] ~\\
Then, since $m\pi<b_{m}\leq\tilde{b}_{m}$, $\delta(m\pi)=1$ and
$\tilde{b}_{m}$ is a maximum of $\delta$, it follows that $\delta(b_{m})>1$.
And, since $\tilde{b}_{m+1}\leq b_{m+1}<(m+2)\pi$, $\delta((m+2)\pi)=1$
and $\tilde{b}_{m+1}$ is a minimum of $\delta$, we have $\delta(b_{m})<1$.
So $\left(\delta(b_{m})-\delta(b_{m+1})\right)a\geq0$, which gives
the claim, since $\mu\geq0$.
\item [{\textbf{2.~case~$a<0$~and~$m$~odd:}}] ~\\
Then, since $m\pi<b_{m}\leq\tilde{b}_{m}$, $\delta(m\pi)=-1$ and
$\tilde{b}_{m}$ is a minimum of $\delta$, it follows that $\delta(b_{m})<-1$.
And, since $\tilde{b}_{m+1}\leq b_{m+1}<(m+2)\pi$, $\delta((m+2)\pi)=-1$
and $\tilde{b}_{m+1}$ is a maximum of $\delta$, we have $\delta(b_{m})>-1$.
So $\left(\delta(b_{m})-\delta(b_{m+1})\right)a\geq0$, which gives
the claim, since \textbf{$\mu\geq0$}.
\item [{\textbf{3.~case~$a\geq0$~and~$m$~odd:}}] ~\\
Here $b_{m}\leq2\hat{b}_{\frac{m-1}{2}}<s_{m+1}\leq b_{m+1}$. Let\[
s:=\begin{cases}
t_{m} & ,\mbox{if }\hat{\mu}\left(\frac{t_{m}}{2}\right)>\hat{\mu}\left(\frac{s_{m+1}}{2}\right)\\
s_{m+1} & ,\mbox{if }\hat{\mu}\left(\frac{t_{m}}{2}\right)\leq\hat{\mu}\left(\frac{s_{m+1}}{2}\right)\end{cases}.\]
Then, since $2\hat{b}_{\frac{m-1}{2}}$ is a minimum of $\hat{\mu}\left(\frac{b}{2}\right)$
on $\left(m\pi,(m+2)\pi\right)$ and monotone on the left and right
of $2\hat{b}_{\frac{m-1}{2}}$: $\max_{b\in[t_{m},s_{m+1}]}\hat{\mu}\left(\frac{b}{2}\right)=\hat{\mu}\left(\frac{s}{2}\right)$
and therefore\begin{align*}
-(s_{m+1}-t_{m})\mu(t_{m}) & =-\left(s_{m+1}-t_{m}\right)\mu(s)\\
 & =-\left(s_{m+1}-t_{m}\right)\left((1-a)\tilde{\mu}(s)+a\hat{\mu}\left(\frac{s}{2}\right)\right)\\
 & \leq-\left(s_{m+1}-t_{m}\right)a\hat{\mu}\left(\frac{s}{2}\right)\\
 & \leq-a\int_{t_{m}}^{s_{m+1}}\hat{\mu}\left(\frac{b}{2}\right)db.\end{align*}
Notice that $\hat{\mu}(\frac{b}{2})$ has no singularity at $(m+1)\pi$.
Now \eqref{eq:l(t_m)-l(s_m+1)} becomes, since $\mu(b)\geq a\hat{\mu}\left(\frac{b}{2}\right)$:\begin{align*}
l(t_{m})-l(s_{m+1}) & \leq-a\int_{b_{m}}^{b_{m+1}}\hat{\mu}\left(\frac{b}{2}\right)db-\left(\delta(b_{m})-\delta(b_{m+1})\right)a\\
 & =-a\int_{b_{m}}^{b_{m+1}}\left(\frac{b/2}{\cos^{2}\frac{b}{2}}+\tan\frac{b}{2}\right)db-\left(\delta(b_{m})-\delta(b_{m+1})\right)a\\
 & =-ab\tan\frac{b}{2}|_{b_{m}}^{b_{m+1}}-\left(\delta(b_{m})-\delta(b_{m+1})\right)a\\
 & =a\left(\delta(b_{m+1})-b_{m+1}\tan\frac{b_{m+1}}{2}-\left(\delta(b_{m})-b_{m}\tan\frac{b_{m}}{2}\right)\right).\end{align*}
Further\begin{align*}
\delta(b)-b\tan\frac{b}{2} & =\cos b+\frac{b}{2}\sin b-b\tan\frac{b}{2}\\
 & =b\sin\frac{b}{2}\left(\cos\frac{b}{2}-\frac{1}{\cos\frac{b}{2}}\right)-2\sin^{2}\frac{b}{2}+1\\
 & =-\left(b\sin^{2}\frac{b}{2}\tan\frac{b}{2}+2\sin^{2}\frac{b}{2}\right)+1\\
 & =-2\sin^{2}\frac{b}{2}\left(1+\frac{b}{2}\tan\frac{b}{2}\right)+1.\end{align*}
Since $b_{m}\leq2\hat{b}_{\frac{m-1}{2}}\leq b_{m+1}$ and $1+\hat{b}_{\frac{m-1}{2}}\tan\hat{b}_{\frac{m-1}{2}}=0$
and $b\mapsto1+\frac{b}{2}\tan\frac{b}{2}$ is increasing on $(m\pi,(m+2)\pi)$,
we conclude:\begin{eqnarray*}
\lefteqn{l(t_{m})-l(s_{m+1})\mbox{}}\\
 & \leq & 2a\left(\sin^{2}\frac{b_{m}}{2}\left(1+\frac{b_{m}}{2}\tan\frac{b_{m}}{2}\right)-\sin^{2}\frac{b_{m+1}}{2}\left(1+\frac{b_{m+1}}{2}\tan\frac{b_{m+1}}{2}\right)\right)\\
 & \leq & 0.\end{eqnarray*}

\item [{\textbf{4.~case~$a<0$~and~$m$~even:}}] ~\\
This is very similar to the third case, if we use $\tilde{\mu}$.
Here $b_{m}\leq2\tilde{b}_{\frac{m}{2}}<s_{m+1}\leq b_{m+1}$. Let\[
s:=\begin{cases}
t_{m} & ,\mbox{ if }\tilde{\mu}\left(\frac{t_{m}}{2}\right)>\tilde{\mu}\left(\frac{s_{m+1}}{2}\right)\\
s_{m+1} & ,\mbox{ if }\tilde{\mu}\left(\frac{t_{m}}{2}\right)\leq\tilde{\mu}\left(\frac{s_{m+1}}{2}\right)\end{cases}.\]
Then, since $2\tilde{b}_{\frac{m}{2}}$ is a minimum of $\tilde{\mu}\left(\frac{b}{2}\right)$
on $\left(m\pi,(m+2)\pi\right)$ and monotone on the left and right
of $2\tilde{b}_{\frac{m}{2}}$: $\max_{b\in[t_{m},s_{m+1}]}\tilde{\mu}(\frac{b}{2})=\tilde{\mu}(\frac{s}{2})$
and therefore\begin{align*}
-(s_{m+1}-t_{m})\mu(t_{m}) & =-\left(s_{m+1}-t_{m}\right)\mu(s)\\
 & =-\left(s_{m+1}-t_{m}\right)\left((1+a)\tilde{\mu}(s)-a\tilde{\mu}\left(\frac{s}{2}\right)\right)\\
 & \leq\left(s_{m+1}-t_{m}\right)a\tilde{\mu}\left(\frac{s}{2}\right)\\
 & \leq a\int_{t_{m}}^{s_{m+1}}\tilde{\mu}\left(\frac{b}{2}\right)db.\end{align*}
Notice that $\tilde{\mu}(\frac{b}{2})$ has no singularity at $(m+1)\pi$.
Now \eqref{eq:l(t_m)-l(s_m+1)} becomes, since $\mu(b)\geq-a\tilde{\mu}(b/2)$:\begin{align*}
l(t_{m})-l(s_{m+1}) & \leq a\left(\int_{b_{m}}^{b_{m+1}}\tilde{\mu}\left(\frac{b}{2}\right)db-\delta(b_{m})+\delta(b_{m+1})\right)\\
 & =a\left(\int_{b_{m}}^{b_{m+1}}\left(\frac{b/2}{\sin^{2}\frac{b}{2}}-\cot\frac{b}{2}\right)db-\delta(b_{m})+\delta(b_{m+1})\right)\\
 & =a\left(-b\cot\frac{b}{2}|_{b_{m}}^{b_{m+1}}-\delta(b_{m})+\delta(b_{m+1})\right)\\
 & =a\left(b_{m}\cot\frac{b_{m}}{2}-\delta(b_{m})-(b_{m+1}\cot\frac{b_{m+1}}{2}-\delta(b_{m+1}))\right).\end{align*}
Further\begin{align*}
b\cot\frac{b}{2}-\delta(b) & =-\cos b-\frac{b}{2}\sin b+b\cot\frac{b}{2}\\
 & =-b\cos\frac{b}{2}\left(\sin\frac{b}{2}-\frac{1}{\sin\frac{b}{2}}\right)-2\cos^{2}\frac{b}{2}+1\\
 & =-2\cos^{2}\frac{b}{2}\left(1-\frac{b}{2}\cot\frac{b}{2}\right)+1.\end{align*}
Since $b_{m}\leq2\tilde{b}_{\frac{m}{2}}\leq b_{m+1}$ and $1-\tilde{b}_{\frac{m}{2}}\cot\tilde{b}_{\frac{m}{2}}=0$
and $b\mapsto1-\frac{b}{2}\cot\frac{b}{2}$ is increasing on $(m\pi,(m+2)\pi)$,
we conclude:\begin{eqnarray*}
\lefteqn{l(t_{m})-l(s_{m+1})}\\
 & \leq & -2a\left(\cos^{2}\frac{b_{m}}{2}\left(1-\frac{b_{m}}{2}\cot\frac{b_{m}}{2}\right)-\cos^{2}\frac{b_{m+1}}{2}\left(1-\frac{b_{m+1}}{2}\cot\frac{b_{m+1}}{2}\right)\right)\\
 & \leq & 0.\end{eqnarray*}
If $t_{0}\in(0,\pi)$ denotes the solution of \eqref{eq:solution},
it remains to show that $l(t_{0})<l(s_{1})$. The above calculations
can be extended to this, if one sets $b_{0}=0$. In detail we have\begin{eqnarray*}
l(t_{0}) & = & \int_{0}^{t_{0}}l'(b)db+l(0)\\
 & = & t_{0}\mu(t_{0})-\int_{0}^{t_{0}}\mu(b)db+1-a,\end{eqnarray*}
so that\[
l(t_{0})-l(s_{1})=-(s_{1}-t_{0})\mu(s_{1})-\int_{0}^{t_{0}}\mu(b)db-\int_{s_{1}}^{b_{1}}\mu(b)db-\left(1-\delta(b_{1})\right)a.\]
Again from the previous lemma one gets $l(t_{0})<l(s_{1})$ if $a\geq0$,
since $1-\delta(b_{1})>0$ and $\mu>0$. If $a<0$, then since $\tilde{\mu}$
is increasing on $[0,\pi)$:\begin{eqnarray*}
-(s_{1}-t_{0})\mu(s_{1}) & = & -(s_{1}-t_{0})\left((1+a)\tilde{\mu}(s_{1})-a\tilde{\mu}\left(\frac{s_{1}}{2}\right)\right)\\
 & \leq & (s_{1}-t_{0})a\tilde{\mu}\left(\frac{s_{1}}{2}\right)\\
 & \leq & a\int_{t_{0}}^{s_{1}}\tilde{\mu}\left(\frac{b}{2}\right)db.\end{eqnarray*}
So we get\begin{eqnarray*}
l(t_{0})-l(s_{1}) & \leq & a\int_{0}^{b_{1}}\tilde{\mu}\left(\frac{b}{2}\right)db-\left(1-\delta(b_{1})\right)a\\
 & = & a\left(-b\cot\frac{b}{2}|_{0}^{b_{1}}-1+\delta(b_{1})\right)\\
 & = & a\left(\delta(b_{1})-b_{1}\cot\frac{b_{1}}{2}+1\right)\\
 & = & 2a\cos^{2}\frac{b_{1}}{2}\left(1-\frac{b_{1}}{2}\cot\frac{b_{1}}{2}\right)\\
 & \leq & 0,\end{eqnarray*}
since $1-\frac{b}{2}\cot\frac{b}{2}\geq0$ on $[0,2\pi)$. 
\end{lyxlist}
\end{proof}
\begin{thm}
\label{thm:x1=00003Dx}Given $(x,0),(x,u)\in\mathbb{R}^{n}\times\mathbb{R},\; x\neq0,\; u>0$,
there are finitely many geodesics connecting $(x,0),\;(x,u)$. Namely:
\begin{lyxlist}{00.}
\item [{\textbf{1.}}] For any solution $b$ of \begin{equation}
\frac{u}{\left|x\right|^{2}}=\hat{\mu}\left(\frac{b}{2}\right)\label{eq:hat mu solution}\end{equation}
there is a geodesic $\gamma^{b}$ given by \eqref{eq:geo1} with square
of length\begin{equation}
L(\gamma^{b})^{2}=\frac{b^{2}}{\cos^{2}\frac{b}{2}}\left|x\right|^{2}=2bu-2\left|x\right|^{2}b\tan\frac{b}{2},\end{equation}
which strictly increases with $b$.
\item [{\textbf{2.}}] If $u\geq\left|x\right|^{2}\pi$, then for each $m\in\mathbb{N}$
with $u\geq m\pi\left|x\right|^{2}$ there are geodesics $\gamma^{b,c}$
given by \eqref{eq:geo1} with $b=2m\pi$, $\left|c\right|=2\sqrt{m\pi u-(m\pi)^{2}\left|x\right|^{2}}$,
with lengths \begin{equation}
L(\gamma^{b,c})=2\sqrt{m\pi u}.\end{equation}

\end{lyxlist}
Moreover the shortest geodesic and therefore the Carnot-Carathéodory
distance is given by the unique solution $b$ of \eqref{eq:hat mu solution}
in the interval $(0,\pi)$, with this $b$ we have\begin{equation}
d_{CC}((x,0),(x,u))=\frac{b}{\cos\frac{b}{2}}|x|=\sqrt{2bu-2\left|x\right|^{2}b\tan\frac{b}{2}}.\end{equation}

\end{thm}
\begin{figure}[ht]
\centering
\mypsteps{\includegraphics{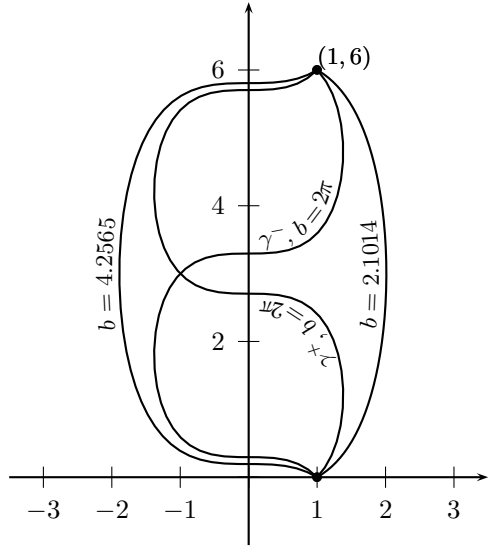}}{
\psset{yunit=.9,xunit=.9,VarStep=true,algebraic=true}
\pspicture*(-3.5,-1)(3.5,7)
\def\geo#1#2#3{(#2*sin(#1*t)/#1+#3*cos(#1*t))|(#2^2*(t/2-sin(2*#1*t)/4/#1)/#1+#3*#2*sin(#1*t)^2/#1+#3^2*#1*(t/2+sin(2*#1*t)/4/#1))} 
\def\geoplot#1#2#3#4{
\psdots(#1,0)(#2,#3)
\rput[lb](#2,#3){\small{$(#2,#3)$}}
\pstextpath[c](0,.1){
\parametricplot{0}{1}{\geo{#4}{(#2*#4/sin(#4)-#1*#4*cos(#4)/sin(#4))}{#1}}}
{\small{$b=#4$}}}
\geoplot{1} 1 {6} {2.1014}
\geoplot{1} 1 {6} {4.2565}
\pstextpath[c](-1.9,.1){\parametricplot{0}{1}{\geo{6.2832}{5.99331}{1.0}}}{\small{$\gamma^+,\;b=2\pi$}}
\pstextpath[c](1.9,.1){\parametricplot{0}{1}{\geo{6.2832}{(-5.99331)}{1.0}}}{\small{$\gamma^-,\;b=2\pi$}}
\psaxes[ticks=all,dx=1,Dx=1,dy=2,Dy=2]{->}(0,0)(-3.5,-1)(3.5,7)
\endpspicture
}
\caption{Geodesics joining $(1,0)$ and $(1,6)$ with different $b\not\in\mathbb Z^\ast$ vs. the two geodesics with $b=2\pi$ and $c=\pm\sqrt{2bu-x_1^2b^2}$.}
\end{figure}

\begin{proof}
Notice that in this case $a=1$. We will consider two cases:
\begin{lyxlist}{00}
\item [{\textbf{1.~case~$b\in\pi\mathbb{N}$:}}] ~\\
Then since $\gamma_{(1)}^{b,c}(1)=x$, there exists $m\in\mathbb{N}$
with $b=2m\pi$. Then $\gamma_{(2)}^{b,c}(1)=u$ is equivalent to
\begin{align*}
u & =\left(\left|c\right|^{2}+(2m\pi)^{2}\left|x\right|^{2}\right)\frac{1}{4m\pi}.\end{align*}
So if $u\geq m\pi\left|x\right|^{2}$, there is a geodesic given by
\eqref{eq:geo1} with \[
b=2m\pi,\quad\left|c\right|=2\sqrt{m\pi u-(m\pi)^{2}\left|x\right|^{2}},\]
with length:\[
L^{2}(\gamma^{b,c})=\left|c\right|^{2}+\left|x\right|^{2}b^{2}=4m\pi u\]
This gives the second part of the theorem.
\item [{\textbf{2.~case~$b\not\in\pi\mathbb{N}$:}}] ~\\
In this case we have\[
\mu(b)=\hat{\mu}\left(\frac{b}{2}\right).\]
This is very similar to the situation in Theorem \ref{thm:allg geo},
but much easier, since we are only dealing with $\hat{\mu}$. From
Lemma \ref{lem:hat mu} we know that the equation \[
\frac{u}{\left|x\right|^{2}}=\hat{\mu}\left(\frac{b}{2}\right)\]
has at most two solutions in an interval $\left((2m+1)\pi,(2m+3)\pi\right),\; m\in\mathbb{N}_{0}$,
each defining a geodesic. If $s$ is a solution of this equation,
the square of the length of the corresponding geodesic $\gamma$ is
given by\begin{align*}
L^{2}(\gamma) & =2l(s)\left|x\right|^{2}\\
 & =2\frac{s^{2}}{\sin^{2}s}\left(1-\cos s\right)\left|x\right|^{2}\\
 & =\frac{s^{2}}{\cos^{2}\frac{s}{2}}\left|x\right|^{2}.\end{align*}
If $s\in((2m+1)\pi,(2m+3)\pi),\; m\in\mathbb{N}_{0}$ , we have\begin{align}
l(s) & =\int_{2\hat{b}_{m}}^{s}l'(b)db+l\left(2\hat{b}_{m}\right)\nonumber \\
 & =\int_{2\hat{b}_{m}}^{s}b\mu'(b)db+l\left(2\hat{b}_{m}\right)\nonumber \\
 & =\int_{2\hat{b}_{m}}^{s}\frac{b}{2}\hat{\mu}'\left(\frac{b}{2}\right)db+l\left(2\hat{b}_{m}\right)\nonumber \\
 & =s\hat{\mu}\left(\frac{s}{2}\right)-2\hat{b}_{m}\hat{\mu}\left(\hat{b}_{m}\right)-\int_{2\hat{b}_{m}}^{s}\hat{\mu}\left(\frac{b}{2}\right)db+l\left(2\hat{b}_{m}\right)\label{eq:hat l s=..}\end{align}
But now \begin{align}
l(2\hat{b}_{m})-2\hat{b}_{m}\hat{\mu}(\hat{b}_{m}) & =\frac{2b_{m}^{2}}{\cos^{2}\hat{b}_{m}}-2\hat{b}_{m}\left(\frac{\hat{b}_{m}}{\cos^{2}\hat{b}_{m}}+\tan\hat{b}_{m}\right)\nonumber \\
 & =-2\hat{b}_{m}\tan\hat{b}_{m}\nonumber \\
 & =2.\label{eq:hat l-b mu=2}\end{align}
First let $s\leq t$ be two solutions in an interval $\left((2m+1)\pi,(2m+3)\pi\right),\; m\in\mathbb{N}_{0}$.
Then by \eqref{eq:hat l s=00003D..} and \eqref{eq:hat l-b mu=00003D2}\begin{align*}
l(s)-l(t) & =s\hat{\mu}(\frac{s}{2})-\int_{2\hat{b}_{m}}^{s}\hat{\mu}\left(\frac{b}{2}\right)db-t\hat{\mu}\left(\frac{t}{2}\right)+\int_{2\hat{b}_{m}}^{t}\hat{\mu}\left(\frac{b}{2}\right)db\\
 & =-(t-s)\hat{\mu}\left(\frac{s}{2}\right)+\int_{s}^{t}\hat{\mu}\left(\frac{b}{2}\right)db\\
 & \leq0.\end{align*}
Now let $t_{m}\in((2m+1)\pi,(2m+3)\pi)$ and $s_{m+1}\in((2m+3)\pi,(2m+5)\pi)$
be two solutions, where $t_{m}$ is the rightmost and $s_{m+1}$ is
the leftmost solution in the corresponding interval. Then by \eqref{eq:hat l s=00003D..}
and \eqref{eq:hat l-b mu=00003D2}\begin{eqnarray*}
\lefteqn{l(t_{m})-l(s_{m+1})}\\
 & = & t_{m}\hat{\mu}\left(\frac{t_{m}}{2}\right)-\int_{2\hat{b}_{m}}^{t_{m}}\hat{\mu}\left(\frac{b}{2}\right)db-s_{m+1}\hat{\mu}\left(\frac{s_{m+1}}{2}\right)+\int_{2\hat{b}_{m+1}}^{s_{m+1}}\hat{\mu}\left(\frac{b}{2}\right)db\\
 & = & -\left(s_{m+1}-t_{m}\right)\hat{\mu}\left(\frac{t_{m}}{2}\right)-\int_{2\hat{b}_{m}}^{t_{m}}\hat{\mu}\left(\frac{b}{2}\right)db-\int_{s_{m+1}}^{2\hat{b}_{m+1}}\hat{\mu}\left(\frac{b}{2}\right)db\\
 & \leq & 0,\end{eqnarray*}
since $2\hat{b}_{m}\leq t_{m}$ and $s_{m+1}\leq2\hat{b}_{m+1}$.
\\
If $t\in(0,\pi)$ is the solution of \eqref{eq:hat mu solution},
we have to check, that $l(t)<l(s_{0})$. As above, we have \begin{alignat*}{1}
l(t) & =\int_{0}^{t}l'(b)db+l\left(0\right)\\
 & =t\hat{\mu}\left(\frac{t}{2}\right)-\int_{0}^{t}\hat{\mu}\left(\frac{b}{2}\right)db+2,\end{alignat*}
so that\begin{alignat*}{1}
l(t)-l(s_{0}) & =t\hat{\mu}\left(\frac{t}{2}\right)-\int_{0}^{t}\hat{\mu}\left(\frac{b}{2}\right)db-s_{0}\hat{\mu}\left(\frac{s_{0}}{2}\right)+\int_{2\hat{b}_{0}}^{s_{0}}\hat{\mu}\left(\frac{b}{2}\right)db\\
 & =-(s_{0}-t)\hat{\mu}\left(\frac{t}{2}\right)-\int_{0}^{t}\hat{\mu}\left(\frac{b}{2}\right)db-\int_{s_{0}}^{2\hat{b}_{0}}\hat{\mu}\left(\frac{b}{2}\right)db\\
 & \leq0.\end{alignat*}

\end{lyxlist}
At last we have to compare the length of the shortest geodesic of
case 1 (if there is one) with that of case 2. Let $b\in(0,\pi)$ be
the solution of \eqref{eq:hat mu solution}, then since $\tan\frac{b}{2}\geq0$:\begin{align*}
L^{2}(\gamma^{b}) & =2bu-2\left|x\right|^{2}b\tan\frac{b}{2}\\
 & \leq2\pi u\\
 & \leq L^{2}(\gamma^{2\pi,c}).\end{align*}

\end{proof}
\begin{thm}
\label{thm:x1=00003D-x}Given $(x,0),(-x,u)\in\mathbb{R}^{n}\times\mathbb{R},\; x\neq0,\; u>0$,
there are finitely many geodesics connecting $(x,0),\;(-x,u)$. Namely:
\begin{lyxlist}{00.}
\item [{\textbf{1.}}] For any solution $b$ of \begin{equation}
\frac{u}{\left|x\right|^{2}}=\tilde{\mu}\left(\frac{b}{2}\right)\label{eq:tilde mu solution}\end{equation}
there is a geodesic $\gamma^{b}$ given by \eqref{eq:geo1} with square
of length\begin{equation}
L(\gamma^{b})^{2}=\frac{b^{2}}{\sin^{2}\frac{b}{2}}\left|x\right|^{2}=2bu+\left|x\right|^{2}b\cot\frac{b}{2},\end{equation}
which strictly increases with $b$.
\item [{\textbf{2.}}] If $2u\geq\left|x\right|^{2}\pi$, then for each
$m\in\mathbb{N}_{0}$ with $2u\geq(2m+1)\pi\left|x\right|^{2}$ there
are geodesics $\gamma^{b,c}$ given by \eqref{eq:geo1} with $b=(2m+1)\pi$,
$\left|c\right|=\sqrt{2(2m+1)\pi u-(2m+1)^{2}\pi^{2}\left|x\right|^{2}}$
and lengths \begin{equation}
L(\gamma^{b,c})=\sqrt{2(2m+1)\pi u}.\end{equation}

\end{lyxlist}
Moreover the shortest geodesic and therefore the Carnot-Carathéodory
distance is given by the unique solution $b$ of \eqref{eq:tilde mu solution}
in the interval $(0,\pi)$, if $2u<\left|x\right|^{2}\pi$. With this
$b$ we have:\begin{equation}
d_{CC}((x,0),(-x,u))=\frac{b}{\sin\frac{b}{2}}|x|=\sqrt{2bu+\left|x\right|^{2}b\cot\frac{b}{2}}.\end{equation}
And, if $2u\geq\left|x\right|^{2}\pi$, then the Carnot-Carathéodory
distance is given by\[
d_{CC}\left((x,0),(-x,u)\right)=\sqrt{2\pi u}.\]

\end{thm}
\begin{figure}[ht]
\centering
\mypsteps{\includegraphics{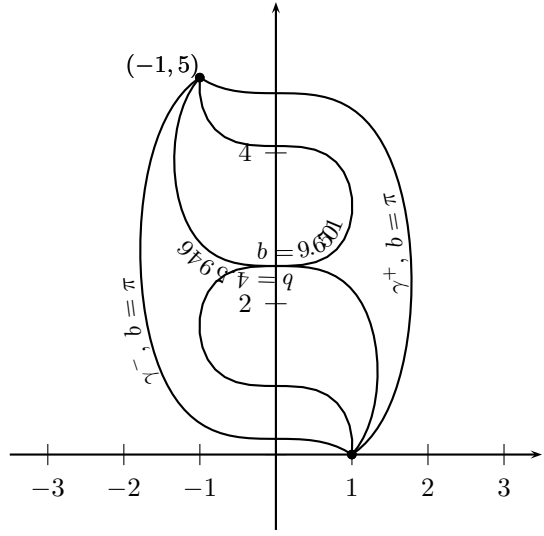}}{
\psset{yunit=1,xunit=1,VarStep=true,algebraic=true}
\pspicture*(-3.5,-1)(3.5,6)
\def\geo#1#2#3{(#2*sin(#1*t)/#1+#3*cos(#1*t))|(#2^2*(t/2-sin(2*#1*t)/4/#1)/#1+#3*#2*sin(#1*t)^2/#1+#3^2*#1*(t/2+sin(2*#1*t)/4/#1))} 
\def\geoplot#1#2#3#4{
\psdots(#1,0)(#2,#3)
\rput[rb](#2,#3){\small{$(#2,#3)$}}
\pstextpath[c](0.5,.1){
\parametricplot{0}{1}{\geo{#4}{(#2*#4/sin(#4)-#1*#4*cos(#4)/sin(#4))}{#1}}}
{\small{$b=#4$}}}
\geoplot{1}{-1}{5} {4.5946}
\geoplot{1}{-1}{5} {9.6501}
\pstextpath[c](-0.5,.1){\parametricplot{0}{1}{\geo{3.1416}{4.6418}{1.0}}}{\small{$\gamma^+,\;b=\pi$}}
\pstextpath[c](0,.1){\parametricplot{0}{1}{\geo{3.1416}{(-4.6418)}{1.0}}}{\small{$\gamma^-,\;b=\pi$}}
\psaxes[ticks=all,dx=1,Dx=1,dy=2,Dy=2]{->}(0,0)(-3.5,-1)(3.5,6)
\endpspicture
}
\caption{Geodesics joining $(1,0)$ and $(1,5)$ with different $b\not\in\mathbb Z^\ast$ vs. two geodesics with $b=\pi$ and $c=\pm\sqrt{2bu-x_1^2b^2}$.}
\end{figure}

\begin{proof}
Notice that in this case $a=-1$. The first part is the same proof
as above, if one replaces $\hat{\mu}$, $\hat{b}$ with $\tilde{\mu}$,
$\tilde{b}$. In particular:
\begin{lyxlist}{00}
\item [{\textbf{1.~case~$b\in\pi\mathbb{N}$:}}] ~\\
Then since $\gamma_{(1)}(1)=x$, there exists $m\in\mathbb{N}_{0}$
with $b=(2m+1)\pi$. Then ${\gamma_{(2)}(1)=u}$ is equivalent to
\begin{align*}
u & =\left(\left|c\right|^{2}+(2m+1)^{2}\pi^{2}\left|x\right|^{2}\right)\frac{1}{2(2m+1)\pi}.\end{align*}
So if $2u\geq(2m+1)\pi\left|x\right|^{2}$, there are geodesics $\gamma^{b,c}$
given by \eqref{eq:geo1} with ${b=(2m+1)\pi}$, $\left|c\right|=\sqrt{2(2m+1)\pi u-(2m+1)^{2}\pi^{2}\left|x\right|^{2}}$,
with length\[
L^{2}(\gamma^{b,c})=\left|c\right|^{2}+\left|x\right|^{2}b^{2}=2(2m+1)\pi u\]
This gives the second part of the theorem.
\item [{\textbf{2.~case~$b\not\in\pi\mathbb{N}$:}}] ~\\
In this case we have:\[
\mu(b)=\tilde{\mu}\left(\frac{b}{2}\right).\]
This is again very similar to the situation in Theorem \ref{thm:allg geo},
but a lot simpler, since we are only dealing with $\tilde{\mu}$.
From Lemma \ref{lem:tilde mu} we know that the equation \begin{equation}
\frac{u}{\left|x\right|^{2}}=\tilde{\mu}\left(\frac{b}{2}\right)\label{eq:equation tilde mu}\end{equation}
has at most two solutions in an interval $(2m\pi,(2m+2)\pi),\; m\in\mathbb{N}$,
each defining an geodesic. If $s$ is a solution of this equation,
the square of the length of the corresponding geodesic $\gamma$ is
given by\begin{align*}
L^{2}(\gamma) & =2l(s)\left|x\right|^{2}\\
 & =2\frac{s^{2}}{\sin^{2}s}\left(1+\cos s\right)\left|x\right|^{2}\\
 & =\frac{s^{2}}{\sin^{2}\frac{s}{2}}\left|x\right|^{2}.\end{align*}
If $s\in(2m\pi,(2m+2)\pi),\; m\in\mathbb{N}$, we have:\begin{align*}
l(s) & =\int_{2\tilde{b}_{m}}^{s}l'(b)db+l\left(2\tilde{b}_{m}\right)\\
 & =\int_{2\tilde{b}_{m}}^{s}b\mu'(b)db+l\left(2\tilde{b}_{m}\right)\\
 & =\int_{2\tilde{b}_{m}}^{s}\frac{b}{2}\tilde{\mu}'\left(\frac{b}{2}\right)db+l\left(2\tilde{b}_{m}\right)\\
 & =s\tilde{\mu}\left(\frac{s}{2}\right)-2\tilde{b}_{m}\tilde{\mu}\left(\tilde{b}_{m}\right)-\int_{2\tilde{b}_{m}}^{s}\tilde{\mu}\left(\frac{b}{2}\right)db+l\left(2\tilde{b}_{m}\right)\end{align*}
But now \begin{align*}
l(2\tilde{b}_{m})-2\tilde{b}_{m}\tilde{\mu}(\tilde{b}_{m}) & =\frac{2\tilde{b}_{m}^{2}}{\sin^{2}\tilde{b}_{m}}-2\tilde{b}_{m}\left(\frac{\tilde{b}_{m}}{\sin^{2}\tilde{b}_{m}}-\cot\tilde{b}_{m}\right)\\
 & =2\tilde{b}_{m}\cot\tilde{b}_{m}\\
 & =2.\end{align*}
First let $s\leq t$ be two solutions of \eqref{eq:equation tilde mu}
in an interval $(2m\pi,(2m+2)\pi),\; m\in\mathbb{N}$; then\begin{align*}
l(s)-l(t) & =s\tilde{\mu}\left(\frac{s}{2}\right)-\int_{2\tilde{b}_{m}}^{s}\tilde{\mu}\left(\frac{b}{2}\right)db-t\tilde{\mu}\left(\frac{t}{2}\right)+\int_{2\tilde{b}_{m}}^{t}\tilde{\mu}\left(\frac{b}{2}\right)db\\
 & =-(t-s)\tilde{\mu}\left(\frac{s}{2}\right)+\int_{s}^{t}\tilde{\mu}\left(\frac{b}{2}\right)db\\
 & \leq0.\end{align*}
Now let $t_{m}\in(2m\pi,(2m+2)\pi)$ and $s_{m+1}\in((2m+2)\pi,(2m+4)\pi)$
be two solutions, then\begin{eqnarray*}
\lefteqn{l(t_{m})-l(s_{m+1})}\\
 & = & t_{m}\tilde{\mu}\left(\frac{t_{m}}{2}\right)-\int_{2\tilde{b}_{m}}^{t_{m}}\tilde{\mu}\left(\frac{b}{2}\right)db-s_{m+1}\tilde{\mu}\left(\frac{s_{m+1}}{2}\right)+\int_{2\tilde{b}_{m+1}}^{s_{m+1}}\tilde{\mu}\left(\frac{b}{2}\right)db\\
 & = & -(s_{m+1}-t_{m})\tilde{\mu}\left(\frac{t_{m}}{2}\right)-\int_{2\tilde{b}_{m}}^{t_{m}}\tilde{\mu}\left(\frac{b}{2}\right)db-\int_{s_{m+1}}^{2\tilde{b}_{m+1}}\tilde{\mu}\left(\frac{b}{2}\right)db\\
 & \leq & 0,\end{eqnarray*}
since $2\tilde{b}_{m}\leq t_{m}$ and $s_{m+1}\leq2\tilde{b}_{m}$.
\\
If $t_{0}\in(0,2\pi)$ is the solution in the interval $(0,2\pi)$,
we have to show that $l(t_{0})<l(s_{1})$. But the above calculations
show that\begin{alignat*}{1}
l(t_{0}) & =\int_{0}^{t_{0}}l'(b)db+l\left(0\right)\\
 & =t_{0}\tilde{\mu}\left(\frac{t_{0}}{2}\right)-\int_{0}^{t_{0}}\tilde{\mu}\left(\frac{b}{2}\right)db+2.\end{alignat*}
So we get\begin{alignat*}{1}
l(t_{0})-l(s_{1}) & =t_{0}\tilde{\mu}\left(\frac{t_{0}}{2}\right)-\int_{0}^{t_{0}}\tilde{\mu}\left(\frac{b}{2}\right)db-s_{1}\tilde{\mu}\left(\frac{s_{1}}{2}\right)+\int_{2\tilde{b}_{1}}^{s_{1}}\tilde{\mu}\left(\frac{b}{2}\right)db\\
 & =-(s_{1}-t_{0})\tilde{\mu}\left(\frac{t_{0}}{2}\right)-\int_{0}^{t_{0}}\tilde{\mu}\left(\frac{b}{2}\right)db-\int_{s_{1}}^{2\tilde{b}_{1}}\tilde{\mu}\left(\frac{b}{2}\right)db\\
 & <0.\end{alignat*}

\end{lyxlist}
At last we have to compare the length of the shortest geodesic $\gamma^{\pi,c},\left|c\right|=\sqrt{2\pi u-\pi^{2}\left|x\right|^{2}}$
of case 1 (if there is one) with that of case 2. Let $b\in(0,2\pi)$
be the solution of \eqref{eq:tilde mu solution}. If $b<\pi$, then,
since $\tilde{\mu}$ is monotone increasing on $(0,\pi)$, we get:\begin{align*}
\frac{u}{\left|x\right|^{2}} & =\tilde{\mu}\left(\frac{b}{2}\right)\\
 & <\tilde{\mu}\left(\frac{\pi}{2}\right)\\
 & =\frac{\pi}{2}.\end{align*}
But this means, that $2u<\pi\left|x\right|^{2}$ and therefore there
is no extra geodesic. If now $b\geq\pi$, we have to show that $L^{2}(\gamma^{b})\geq2\pi u$.
We have\begin{alignat*}{1}
\frac{u}{\left|x\right|^{2}} & =\tilde{\mu}\left(\frac{b}{2}\right)\\
 & =\frac{\frac{b}{2}}{\sin^{2}\frac{b}{2}}-\cot\frac{b}{2}\\
 & =\cot\frac{b}{2}\left(\frac{b}{2\sin\frac{b}{2}\cos\frac{b}{2}}-1\right)\\
 & =\frac{b-\sin b}{\sin b}\cot\frac{b}{2},\end{alignat*}
so that\begin{alignat*}{1}
L^{2}\left(\gamma^{b}\right) & =2bu+2\left|x\right|^{2}b\cot\frac{b}{2}\\
 & =2bu\left(1+\frac{\sin b}{b-\sin b}\right)\\
 & =2u\frac{b^{2}}{b-\sin b}\\
 & \geq2\pi u.\end{alignat*}
The last inequality holds since $f(b):=\frac{b^{2}}{b-\sin b}\geq\pi$,
for all $b\in[\pi,2\pi]$. To see this, observe, that $f(\pi)=\pi$
and $f$ is monotone increasing:\begin{alignat*}{1}
f'(b) & =\frac{2b(b-\sin b)-b^{2}(1-\cos b)}{(b-\sin b)^{2}}\\
 & =\frac{b^{2}(1+\cos b)-2b\sin b}{(b-\sin b)^{2}}\\
 & \geq0.\end{alignat*}
This completes the proof, since the geodesics $\gamma^{b,c}$ in case
2 (with parameter $b=\pi$, $\left|c\right|=\sqrt{2\pi u-\pi^{2}x^{2}}$)
are shorter. 
\end{proof}
To close this section, we give a summarizing theorem about our results
of this section: 

\begin{thm}
Let $(x_{1},u_{1}),\;(x,u)\in\mathbb{R}^{n+1}$. If $x_{1}=-x$ and
$2\left|u-u_{1}\right|\geq\pi\left|x\right|^{2}$, then the Carnot-Carathéodory
distance is given by\begin{equation}
d_{CC}\left((x_{1},u_{1}),(x,u)\right)=\sqrt{2\pi\left|u-u_{1}\right|}.\label{eq:dcc x=-x1}\end{equation}
Otherwise the Carnot-Carathéodory distance is given by\begin{subequations}\begin{alignat}{1}
d_{CC}\left((x_{1},u_{1}),(x,u)\right) & =\frac{b}{\sin b}\sqrt{\left|x_{1}\right|^{2}+\left|x\right|^{2}-2x_{1}\cdot x\cos b}\label{eq:dcc1}\\
 & =\sqrt{2b(u-u_{1})+(\left|x_{1}\right|^{2}+\left|x\right|^{2})b\cot b-2x_{1}\cdot x\frac{b}{\sin b}}\label{eq:dcc2}\end{alignat}
\end{subequations}where $b\in(-\pi,\pi)$ is the unique solution
of\begin{equation}
\frac{2(u-u_{1})}{\left|x_{1}\right|^{2}+\left|x\right|^{2}}=\frac{b}{\sin^{2}b}-\cot b+2\frac{x_{1}\cdot x}{\left|x_{1}\right|^{2}+\left|x\right|^{2}}\cdot\frac{1-b\cot b}{\sin b}\label{eq:solution new}\end{equation}
in the interval $(-\pi,\pi)$. 
\end{thm}

\section{Heat kernel}

We now turn our attention to the heat equation, i.e. the Cauchy problem\begin{equation}
\begin{cases}
(\partial_{t}-G)v(t,x,u)=0\\
v(0,x,u)=f(x,u)\end{cases}\label{eq:heat eqn}\end{equation}
 with some function $f\in\mathcal{S}(\mathbb{R}^{n+1})$. Taking the
Fourier-transform in the $u$-variable we see that this is equivalent
to\begin{equation}
\begin{cases}
(\partial_{t}-H_{\lambda})w(t,x,\lambda)=0\\
w(0,x,\lambda)=g(x,\lambda)\end{cases},\label{eq:fourier heat eqn}\end{equation}
where $g(x,\lambda):=f(x,\hat{\lambda})$ denotes the partial Fourier-transform
of $f$ in the $u$-variable and $H_{\lambda}=\Delta_{x}-|\lambda|^{2}|x|^{2}$
denote the Hermite-operator with parameter $\lambda\in\mathbb{R}$,
which is the partial Fourier-transform of $G$. Thanks to the Plancherel
formula, any solution $w$ of \eqref{eq:fourier heat eqn} defines
a solution $v$ to \eqref{eq:heat eqn} by taking the inverse Fourier-transform,
and vice versa.

But the latter equation is well known and we give a brief discussion
of this equation. The solution of this equation is given by $e^{tH_{\lambda}}g$,
where $e^{tH_{\lambda}}$ is a bounded operator defined by the functional
calculus, since $H_{\lambda}$ is a self-adjoint operator. By the
Schwartz kernel theorem $e^{tH_{\lambda}}$ is given by an integral
kernel, i.e. \begin{equation}
e^{tH_{\lambda}}g(x)=\int_{\mathbb{R}^{n}}k_{t}^{\lambda}(x,\xi)g(\xi)d\xi.\end{equation}
And, since $H_{\lambda}$ is hypoelliptic, we have $k_{t}^{\lambda}\in C^{\infty}(\mathbb{R}^{2n})$.

The Hermite-functions\begin{equation}
h_{\alpha}(x):=\frac{2^{n/4}}{\sqrt{\alpha!}}\left(\frac{-1}{2\sqrt{\pi}}\right)^{|\alpha|}e^{\pi|x|^{2}}\left(\frac{\partial}{\partial x}\right)^{\alpha}e^{-2\pi|x|^{2}},\quad x\in\mathbb{R}^{n},\;\alpha\in\mathbb{N}_{0}^{n},\end{equation}
which give an orthonormal basis of $L^{2}(\mathbb{R}^{n})$, are eigenfunctions
of $H_{2\pi}$:\begin{equation}
H_{2\pi}h_{\alpha}=-2\pi\left(2|\alpha|+n\right)h_{\alpha}.\end{equation}
Since\begin{equation}
\left(\Delta-\mu^{2}|x|^{2}\right)f(r\cdot)=r^{2}\left(\left(\Delta-\left(\frac{\mu}{r^{2}}\right)^{2}|\cdot|^{2}\right)f\right)(r\cdot),\end{equation}
we can rescale the Hermite functions\begin{equation}
h_{\alpha}^{\lambda}(x):=\left(\frac{|\lambda|}{2\pi}\right)^{\frac{n}{4}}h_{\alpha}\left(\left(\frac{|\lambda|}{2\pi}\right)^{\frac{1}{2}}x\right)\end{equation}
to get an orthonormal basis $\left(h_{\alpha}^{\lambda}\right)_{\alpha\in\mathbb{N}_{0}^{n}}$
of eigenfunctions of $H_{\lambda}$:\begin{equation}
H_{\lambda}h_{\alpha}^{\lambda}=-|\lambda|\left(2|\alpha|+n\right)h_{\alpha}^{\lambda}.\end{equation}
Using Mehler's formula, we can calculate the kernel $k_{t}^{\lambda}$:\begin{alignat*}{1}
k_{t}^{\lambda}(x,\xi) & =\sum_{|\alpha|\geq0}e^{tH_{\lambda}}h_{\alpha}^{\lambda}(x)h_{\alpha}^{\lambda}(\xi)\\
 & =\sum_{|\alpha|\geq0}e^{-|\lambda|(2|\alpha|+n)t}\left(\frac{|\lambda|}{2\pi}\right)^{\frac{n}{2}}h_{\alpha}\left(\sqrt{\frac{|\lambda|}{2\pi}}x\right)h_{\alpha}\left(\sqrt{\frac{|\lambda|}{2\pi}}\xi\right)\\
 & =\left(\frac{|\lambda|}{2\pi}\right)^{\frac{n}{2}}e^{-n|\lambda|t}\sum_{|\alpha|\geq0}e^{-2|\lambda||\alpha|t}h_{\alpha}\left(\sqrt{\frac{|\lambda|}{2\pi}}x\right)h_{\alpha}\left(\sqrt{\frac{|\lambda|}{2\pi}}\xi\right)\\
 & =\left(\frac{|\lambda|}{2\pi}\right)^{\frac{n}{2}}e^{-n|\lambda|}\left(\frac{2}{1-e^{-4|\lambda|t}}\right)^{\frac{n}{2}}\\
 & \hspace{1cm}\cdot\exp\left(\frac{-\pi\left(1+e^{-4|\lambda|t}\right)\frac{\left|x\right|^{2}+\left|\xi\right|^{2}}{2\pi}|\lambda|+4\pi\frac{x\cdot\xi}{2\pi}|\lambda|e^{-2|\lambda|t}}{1-e^{-4|\lambda|t}}\right)\\
 & =(2\pi)^{-n/2}\left(\frac{|\lambda|}{\sinh(2|\lambda|t)}\right)^{\frac{n}{2}}\\
 & \hspace{1cm}\cdot\exp\left(-\frac{\left|x\right|^{2}+\left|\xi\right|^{2}}{2}|\lambda|\coth(2|\lambda|t)+\frac{|\lambda|}{\sinh(2|\lambda|t)}x\cdot\xi\right)\\
 & =(4\pi t)^{-n/2}\left(\frac{2\lambda t}{\sinh(2\lambda t)}\right)^{\frac{n}{2}}\\
 & \hspace{1cm}\cdot\exp\left(-\frac{1}{4t}\left(2\lambda t\coth(2\lambda t)(\left|x\right|^{2}+\left|\xi\right|^{2})-\frac{4\lambda t}{\sinh(2\lambda t)}x\cdot\xi\right)\right).\end{alignat*}
Now, by taking the Fourier inverse\begin{eqnarray}
\lefteqn{K_{t}(x,\xi,u)=(2\pi)^{-1}\int\limits _{-\infty}^{\infty}k_{t}^{\lambda}(x,\xi)e^{i\lambda u}d\lambda}\label{eq:Kernel}\\
\lefteqn{} & = & (2\pi)^{-1}(4\pi t)^{-n/2}\int\limits _{-\infty}^{\infty}\left(\frac{2\lambda t}{\sinh(2\lambda t)}\right)^{\frac{n}{2}}\nonumber \\
 &  & \hspace{0.5cm}\cdot\exp\left(-\frac{1}{4t}\left(2\lambda t\coth(2\lambda t)(\left|x\right|^{2}+\left|\xi\right|^{2})-\frac{4\lambda t}{\sinh(2\lambda t)}x\cdot\xi\right)\right)e^{i\lambda u}d\lambda\nonumber \\
 & = & (4\pi t)^{-n/2-1}\int\limits _{-\infty}^{\infty}\left(\frac{\lambda}{\sinh\lambda}\right)^{\frac{n}{2}}\nonumber \\
 &  & \hspace{0.5cm}\exp\left(-\frac{1}{4t}\left(\lambda\coth\lambda(\left|x\right|^{2}+\left|\xi\right|^{2})-\frac{2\lambda}{\sinh\lambda}x\cdot\xi-2i\lambda u\right)\right)d\lambda,\nonumber \end{eqnarray}
one gets the heat kernel for the Grušin operator. Now the solution
of \eqref{eq:heat eqn} is given by\begin{equation}
v(x,u)=\int_{\mathbb{R}}\int_{\mathbb{R}^{n}}K_{t}(x,\xi,u-\lambda)f(\xi,\lambda)d\xi d\lambda.\end{equation}

\section{Estimates}

We may now use the results of section \ref{sec:Geometry-introduced-by}
to give some estimates for the heat kernel. We set\begin{alignat}{1}
\lefteqn{}h(x,\xi,u) & :=\int\limits _{-\infty}^{\infty}\left(\frac{\lambda}{\sinh\lambda}\right)^{\frac{n}{2}}\exp\left(-\left(\lambda\coth\lambda(\left|x\right|^{2}+\left|\xi\right|^{2})-\frac{2\lambda}{\sinh\lambda}x\cdot\xi-2i\lambda u\right)\right)d\lambda\label{eq:simple kernel}\\
 & =\int\limits _{-\infty}^{\infty}\left(\frac{\lambda}{\sinh\lambda}\right)^{\frac{n}{2}}\exp\left(-\left(\lambda\coth\lambda-\frac{\lambda}{\sinh\lambda}a\right)R^{2}+2i\lambda u\right)d\lambda\nonumber \end{alignat}
with \[
R:=R(x,\xi):=\sqrt{\left|x\right|^{2}+\left|\xi\right|^{2}}\;\mbox{and}\; a:=a(x,\xi):=\frac{2x\cdot\xi}{R^{2}}\in\left[0,1\right].\]
Then\begin{equation}
K_{t}(x,\xi,u)=(4\pi t)^{-n/2-1}h\left(\frac{x}{2\sqrt{t}},\frac{\xi}{2\sqrt{t}},\frac{u}{4t}\right).\label{eq:h eq kt}\end{equation}
It is also convenient to set\begin{eqnarray}
V(\lambda) & := & \left(\frac{\lambda}{\sinh\lambda}\right)^{\frac{n}{2}}\label{eq:V lambda}\\
\psi(\lambda) & := & \psi(\lambda,a)\label{eq:psi lambda}\\
 & := & \lambda\coth\lambda-\frac{\lambda}{\sinh\lambda}a,\quad\lambda\in\mathbb{R}.\nonumber \end{eqnarray}
One may see that $\psi(ib)=b\cot b-\frac{b}{\sin b}a$, so that if
$b$ is a solution of \eqref{eq:solution new}, then the exponent\[
\left(\lambda\coth\lambda-\frac{\lambda}{\sinh\lambda}a\right)R^{2}-2i\lambda u\]
at $\lambda=ib$ gives exactly the square of the Carnot-Carathéodory
distance $d_{CC}((x,0),(\xi,u))^{2}$. So we expect Gaussian-type
estimates of the form\[
|h(x,\xi,u)|\leq F(x,\xi,u)e^{-d_{CC}((x,0),(\xi,u))^{2}},\]
where $F(x,\xi,u)>0$ is a function depending on $x,\xi,u$. It will
turn out that $F$ has polynomial growth, more precisely \[
F(x,\xi,u)\lesssim(1+d_{CC}((x,0),(\xi,u))^{2})^{\alpha},\]
where $\alpha=\frac{n}{2}-1$, if $n>2$ and $\alpha=0$ if $n\leq2$.
Compared to the euclidean case, where $\alpha=0$, one has some additional
growth, if the dimension is greater than $2$.

The strategy is as follows: Move the line of integration in \eqref{eq:simple kernel}
from the real axis to the line $\mathbb{R}+ib$, where $b$ is near
to the solution $b_{0}$ of \eqref{eq:solution new}.

\begin{lem}
The function\[
V(\lambda):=\left(\frac{\lambda}{\sinh\lambda}\right)^{\frac{n}{2}},\]
where the square root is the principal branch in $\mathbb{C}\setminus(-\infty,0]$,
is holomorphic in $\left\{ z\in\mathbb{C}:|\Im z|<\pi\right\} $.
And for $\nu+ib\in\mathbb{C}$, $-\pi<b<\pi$, one has\begin{eqnarray}
|V(\nu+ib)| & \leq & \left(\frac{b}{\sin b}\right)^{\frac{n}{2}},\label{eq:V bound small}\\
|V(\nu+ib)| & \leq & \left(1+\frac{b^{2}}{\nu^{2}}\right)^{\frac{n}{4}}\left(\frac{\nu}{\sinh\nu}\right)^{\frac{n}{2}}.\label{eq:V bound large}\end{eqnarray}

\end{lem}
\begin{proof}
To show that $V$ is holomorphic, we have to show that \[
\frac{\lambda}{\sinh\lambda}\not\in(-\infty,0],\]
for all $\lambda\in\mathbb{C}$ with $-\pi<\Im\lambda<\pi$. Let $\nu+ib\in\mathbb{C}$,
$|b|<\pi.$ Then \begin{alignat}{1}
\frac{\nu+ib}{\sinh(\nu+ib)} & =\frac{(\nu+ib)\sinh(\nu-ib)}{|\sinh(\nu+ib)|^{2}}\nonumber \\
 & =\frac{\nu\sinh\nu\cos b+b\cosh\nu\sin b+i(b\sinh\nu\cos b-\nu\cosh\nu\sin b)}{|\sinh(\nu+ib)|^{2}}.\label{eq:re,im z/sh z}\end{alignat}
So \begin{alignat*}{1}
\frac{\nu+ib}{\sinh(\nu+ib)}\in\mathbb{R} & \Leftrightarrow b\sinh\nu\cos b=\nu\cosh\nu\sin b\\
 & \Leftrightarrow b\cot b=\nu\coth\nu\vee\nu=0.\end{alignat*}
But $b\cot b\leq1$ and $\nu\coth\nu>1$ for $\nu\neq0$; and $\nu=0$
means \[
\frac{\nu+ib}{\sinh(\nu+ib)}=\frac{b}{\sin b}>0.\]
We can conclude \begin{equation}
\frac{\nu+ib}{\sinh(\nu+ib)}\not\in(-\infty,0],\ \quad\nu\in\mathbb{R},\,-\pi<b<\pi.\end{equation}
The second inequality is easy, since\begin{alignat*}{1}
\left|\frac{\nu+ib}{\sinh(\nu+ib)}\right|^{2} & =\frac{\nu^{2}+b^{2}}{\sinh^{2}\nu+\sin^{2}b}\\
 & \leq\frac{\nu^{2}+b^{2}}{\sinh^{2}\nu}\\
 & =\left(1+\frac{b^{2}}{\nu^{2}}\right)\left(\frac{\nu}{\sinh\nu}\right)^{2}.\end{alignat*}
To verify the first inequality we show that \[
f(\nu):=\left|\frac{\nu+ib}{\sinh(\nu+ib)}\right|^{2}\leq\frac{b^{2}}{\sin^{2}b},\;\mbox{for all }\nu\in\mathbb{R}.\]
This is obviously true for $\nu=0$, and since $f(-\nu)=f(\nu)$,
we may restrict our analysis to the case $\nu\geq0$:\begin{alignat*}{1}
f'(\nu) & =\frac{2\nu(\sinh^{2}\nu+\sin^{2}b)-(\nu^{2}+b^{2})\sinh(2\nu)}{(\sinh^{2}\nu+\sin^{2}b)^{2}}\\
 & =\frac{2\nu\sinh\nu(\sinh\nu-\nu\cosh\nu)+2\nu\sin^{2}b-b^{2}\sinh(2\nu)}{(\sinh^{2}\nu+\sin^{2}b)^{2}}\\
 & \leq0,\end{alignat*}
since $\sinh\nu\leq\nu\cosh\nu$ and $2\nu\sin^{2}b\leq\sinh(2\nu)\sin^{2}b\leq b^{2}\sinh(2\nu)$.
This means that $f$ stays below $\frac{b^{2}}{\sin^{2}b}$ for all
$\nu\in\mathbb{R}$.
\end{proof}
\begin{lem}
The function\begin{subequations}\label{eq:convex psi} \begin{alignat}{1}
\psi(\lambda) & =\lambda\coth\lambda-a\frac{\lambda}{\sinh\lambda}\\
 & =(1-a)\lambda\coth\lambda+a\lambda\tanh\frac{\lambda}{2}\\
 & =(1+a)\lambda\coth\lambda-a\lambda\coth\frac{\lambda}{2}\end{alignat}
\end{subequations}is holomorphic in $\mathbb{C}\setminus\left\{ \pi ki:k\in\mathbb{Z}^{\ast}\right\} $
for $a\neq\pm1$.\\
If $a=1$, then $\psi(\lambda)=\lambda\tanh\frac{\lambda}{2}$ is
holomorphic in $\mathbb{C}\setminus\left\{ (2k+1)\pi i:\: k\in\mathbb{Z}\right\} $;\\
if $a=-1$, then $\psi(\lambda)=\lambda\coth\frac{\lambda}{2}$ is
holomorphic in $\mathbb{C}\setminus\left\{ 2k\pi i:\: k\in\mathbb{Z}^{\ast}\right\} $.

In particular $\psi$ is holomorphic in the strip $\left\{ z\in\mathbb{C}:\;|\Im z|<\pi\right\} $
for any ${-1\leq a\leq1}$.
\end{lem}
\begin{proof}
The holomorphic properties of $\psi$ follow easily from those of
the trigonometric function $\coth,\,\tanh$, and the equalities \eqref{eq:convex psi},
which are easy to verify:\begin{alignat*}{1}
\psi(\lambda) & =\lambda\coth\lambda-a\frac{\lambda}{\sinh\lambda}\\
 & =(1-a)\lambda\coth\lambda+a\lambda\frac{\cosh\lambda-1}{\sinh\lambda}\\
 & =(1-a)\lambda\coth\lambda+a\lambda\tanh\frac{\lambda}{2},\end{alignat*}
and\begin{alignat*}{1}
\psi(\lambda) & =(1+a)\lambda\coth\lambda-a\lambda\frac{\cosh\lambda+1}{\sinh\lambda}\\
 & =(1+a)\lambda\coth\lambda-a\lambda\coth\frac{\lambda}{2}.\end{alignat*}

\end{proof}
\begin{lem}
For $\nu+ib\in\mathbb{C}$, $\left|b\right|<\pi$, $\psi(ib)$ is
real and\begin{equation}
\Re\left(\psi(\nu+ib)\right)\geq\psi(ib),\end{equation}
where \begin{subequations}\label{eq:convex psi ib} \begin{alignat}{1}
\psi(ib) & =b\cot b-a\frac{b}{\sin b}\\
 & =(1-a)b\cot b+a\left(-b\tan\frac{b}{2}\right)\\
 & =(1+a)b\cot b-ab\cot\frac{b}{2}.\end{alignat}
\end{subequations}
\end{lem}
\begin{proof}
\eqref{eq:convex psi ib} is clear. Let $\nu+ib\in\mathbb{C},\;|b|\leq\pi$,
then\begin{alignat*}{1}
\Re\left((\nu+ib)\coth(\nu+ib)\right) & =\frac{1}{2}\frac{\nu\sinh(2\nu)+b\sin(2b)}{\sinh^{2}\nu+\sin^{2}b}\\
 & =\frac{\nu\coth\nu\sinh^{2}\nu+b\cot b\sin^{2}b}{\sinh^{2}\nu+\sin^{2}b}\\
 & =b\cot b+\sinh^{2}\nu\frac{\nu\coth\nu-b\cot b}{\sinh^{2}\nu+\sin^{2}b}\\
 & \geq b\cot b,\end{alignat*}
and for $\left|b\right|\leq\frac{\pi}{2}$:\begin{alignat*}{1}
\Re\left((\nu+ib)\tanh(\nu+ib)\right) & =\frac{1}{2}\frac{\nu\sinh(2\nu)-b\sin(2b)}{\sinh^{2}\nu+\cos^{2}b}\\
 & =\frac{\nu\coth\nu\sinh^{2}\nu-b\tan b\cos^{2}b}{\sinh^{2}\nu+\cos^{2}b}\\
 & =-b\tan b+\sinh^{2}\nu\frac{\nu\coth\nu+b\tan b}{\sinh^{2}\nu+\cos^{2}b}\\
 & \geq-b\tan b.\end{alignat*}
This gives for $\lambda=\nu+ib\in\mathbb{C}$, $\left|b\right|<\pi$:\begin{alignat*}{1}
\Re\left(\psi(\lambda)\right) & =\begin{cases}
\Re\left((1-a)\lambda\coth\lambda+a\lambda\tanh\frac{\lambda}{2}\right) & ,\;\mbox{if }a\geq0\\
\Re\left((1+a)\lambda\coth\lambda-a\lambda\coth\frac{\lambda}{2}\right) & ,\;\mbox{if }a<0\end{cases}\\
 & \geq\begin{cases}
(1-a)b\cot b+a(-b\tan\frac{b}{2}) & ,\;\mbox{if }a\geq0\\
(1+a)b\cot b-ab\cot\frac{b}{2} & ,\;\mbox{if }a<0\end{cases}\\
 & =\psi(ib).\end{alignat*}

\end{proof}
With these preliminary lemmata we can start estimating $h$ by moving
the path of integration to $\lambda\mapsto\lambda+ib$:

\begin{lem}
\label{lem:main estimate}There is a constant $C>0$, such that for
all $x,\xi\in\mathbb{R}^{n}$, $u\in\mathbb{R}$ and $b\in\mathbb{R},\left|b\right|<\pi$
the following estimate holds:\begin{equation}
\left|h(x,\xi,u)\right|\leq C\left(\frac{b}{\sin b}\right)^{\alpha}\exp\left(-2bu-\left(b\cot b-\frac{b}{\sin b}\frac{2x\cdot\xi}{\left|x\right|^{2}+\left|\xi\right|^{2}}\right)\left(\left|x\right|^{2}+\left|\xi\right|^{2}\right)\right),\end{equation}
with $\alpha=\max\left(0,\frac{n}{2}-1\right)$.
\end{lem}
\begin{proof}
We have\begin{alignat*}{1}
\left|h(x,\xi,u)\right| & =\left|\int\limits _{-\infty}^{\infty}V(\lambda)e^{-\psi(\lambda)R^{2}+2i\lambda u}d\lambda\right|\\
 & =\left|\int\limits _{-\infty}^{\infty}V(\lambda+ib)e^{-\psi(\lambda+ib)R^{2}+2i(\lambda+ib)u}d\lambda\right|\\
 & \leq\int\limits _{-\infty}^{\infty}\left|V(\lambda+ib)\right|e^{-\Re\psi(\lambda+ib)R^{2}-2bu}d\lambda\\
 & =e^{-\psi(ib)R^{2}-2bu}\int\limits _{-\infty}^{\infty}\left|V(\lambda+ib)\right|d\lambda.\end{alignat*}
Now we split the remaining integral with $r:=\pi-\left|b\right|$:\begin{alignat*}{1}
\int\limits _{-\infty}^{\infty}\left|V(\lambda+ib)\right|d\lambda & =2\int\limits _{0}^{\infty}\left|V(\lambda+ib)\right|d\lambda\\
 & \lesssim\int\limits _{0}^{r}\left|\frac{b}{\sin b}\right|^{\frac{n}{2}}d\lambda+\int\limits _{r}^{\pi}\left(1+\frac{b^{2}}{\lambda^{2}}\right)^{\frac{n}{4}}d\lambda+\int\limits _{\pi}^{\infty}\left|\frac{\lambda}{\sinh\lambda}\right|^{\frac{n}{2}}d\lambda\\
 & \lesssim r\left(\frac{b}{\sin b}\right)^{\frac{n}{2}}+\int\limits _{r}^{\pi}\frac{1}{\lambda^{n/2}}d\lambda+1\\
 & \lesssim\left(\frac{b}{\sin b}\right)^{\frac{n}{2}-1}+\left(\frac{1}{r}\right)^{\frac{n}{2}-1}+1\\
 & \lesssim\left(\frac{b}{\sin b}\right)^{\alpha},\end{alignat*}
since \begin{equation}
\frac{b}{\sin b}\backsim\frac{1}{\pi-\left|b\right|}=\frac{1}{r}.\label{eq:b/sin b asym}\end{equation}

\end{proof}
Observe that we can assume $u\geq0$, since $h(x,\xi,-u)=h(x,\xi,u)$. 

Setting $b=b_{0}$, where $b_{0}$ parametrizes the corresponding
(shortest) geodesic, would not give good results in every case. If
$b$ tends to $\pi$, the estimate would blow up. To avoid this problem,
just set $b=b_{0}-\epsilon$ with an appropriate $\epsilon>0$. (Since
we assumed $u\geq0$, we have $b_{0}\geq0$.

\begin{prop}
For $\zeta=(x,0),\;\eta=(\xi,u)\in\mathbb{R}^{n+1}$ we have\begin{equation}
\left|h(x,\xi,u)\right|\lesssim\min\left(1+\frac{d_{CC}(\zeta,\eta)}{\left|x+\xi\right|},1+d_{CC}(\zeta,\eta)^{2}\right)^{\alpha}e^{-d_{CC}(\zeta,\eta)^{2}},\end{equation}
with $\alpha=\max(\frac{n}{2}-1,0)$.
\end{prop}
\begin{proof}
First assume that $x\neq-\xi$. Then the shortest geodesic and therefore
the Carnot-Carathéodory distance will be parametrized by $b_{0}\in(\pi,\pi)$.
We claim that \[
\frac{b_{0}^{2}}{\sin^{2}b_{0}}\leq3\left(1+\frac{d_{CC}(\zeta,\eta)^{2}}{\left|x+\xi\right|^{2}}\right).\]
To see this, first assume that $a<0$, i.e. $\xi\cdot x<0$. Then\begin{alignat*}{1}
d_{CC}(\zeta,\eta)^{2} & =\frac{b_{0}^{2}}{\sin^{2}b_{0}}\left(\left|x\right|^{2}+\left|\xi\right|^{2}-2\xi\cdot x\cos b_{0}\right)\\
 & \geq\frac{b_{0}^{2}}{\sin^{2}b_{0}}\left(\left|x\right|^{2}+\left|\xi\right|^{2}+2\xi\cdot x\right)\\
 & =\frac{b_{0}^{2}}{\sin^{2}b_{0}}\left|x+\xi\right|^{2}.\end{alignat*}
If $a\geq0$, i.e. $\xi\cdot x\geq0$, then\begin{alignat*}{1}
d_{CC}(\zeta,\eta)^{2} & \geq\frac{1}{3}\frac{b_{0}^{2}}{\sin^{2}b_{0}}\left(1-a\cos b_{0}\right)\left(\left|x\right|^{2}+2\xi\cdot x+\left|\xi\right|^{2}\right)\\
 & \geq\frac{1}{3}\left(\frac{b_{0}^{2}}{\sin^{2}b_{0}}-a\frac{b_{0}^{2}\cos b_{0}}{\sin^{2}b_{0}}\right)\left|x+\xi\right|^{2}.\end{alignat*}
If $b_{0}>\frac{\pi}{2}$, then $a\frac{b_{0}^{2}}{\sin^{2}b_{0}}\cos b_{0}<0$
and if $b_{0}\geq\frac{\pi}{2}$, then $a\frac{b_{0}^{2}}{\sin^{2}b_{0}}\cos b_{0}\leq\frac{b_{0}^{2}}{\sin^{2}b_{0}}\leq\frac{\pi^{2}}{4}.$
So\begin{alignat*}{1}
d_{CC}(\zeta,\eta)^{2} & \geq\frac{1}{3}\left(\frac{b_{0}^{2}}{\sin^{2}b_{0}}-\frac{\pi^{2}}{4}\right)\left|x+\xi\right|^{2}\\
 & \geq\frac{1}{3}\left(\frac{b_{0}^{2}}{\sin^{2}b_{0}}-3\right)\left|x+\xi\right|^{2}.\end{alignat*}
This proves the claim. Now use Lemma \ref{lem:main estimate} with
$b=b_{0}$ to get\begin{equation}
\left|h(x,\xi,u)\right|\lesssim\left(1+\frac{d_{CC}(\zeta,\eta)}{\left|x+\xi\right|}\right)^{\alpha}e^{-d_{CC}(\zeta,\eta)^{2}}.\end{equation}
But this remains true, if $x=-\xi$ (and $x=\xi=0$), since the right
side becomes $+\infty$ in this case. To prove the remaining estimate\begin{equation}
\left|h(x,\xi,u)\right|\lesssim(1+d_{CC}(\zeta,\eta)^{2})^{\alpha}e^{-d_{CC}(\zeta,\eta)^{2}},\end{equation}
we, again, use lemma \ref{lem:main estimate}. Let $b_{0}\in\left[0,\pi\right]$
be the parameter for the Carnot-Carathéodory distance. As mentioned
above assume that $u\geq0$. If $b_{0}\leq\frac{\pi}{2}$, then $\frac{b_{0}}{\sin b_{0}}\leq\frac{\pi}{2}$,
and one can immediately use Lemma \ref{lem:main estimate} with $b=b_{0}$.
So assume that $b_{0}>\frac{\pi}{2}$. Now set $b=b_{0}-\epsilon$
with $\epsilon<\frac{\pi}{2}$, which will be specified later. 

If $b_{0}\neq\pi$, we have\begin{alignat*}{1}
b\cot b & =b_{0}\cot b_{0}+b\left(\cot b-\cot b_{0}\right)-\epsilon\cot b_{0}\\
 & \geq b_{0}\cot b_{0}-\epsilon\cot b_{0},\end{alignat*}
since $\cot$ is monotone decreasing on $(0,\pi)$; and\begin{alignat*}{1}
-b\tan\frac{b}{2} & =-b_{0}\tan\frac{b_{0}}{2}+b\left(\tan\frac{b_{0}}{2}-\tan\frac{b}{2}\right)+\epsilon\tan\frac{b_{0}}{2}\\
 & \geq-b_{0}\tan\frac{b_{0}}{2}+\epsilon\tan\frac{b_{0}}{2},\end{alignat*}
since $\tan$ is monotone increasing on $(0,\pi)$. Therefore\begin{alignat*}{1}
\psi(ib) & =\begin{cases}
(1-a)b\cot b+a(-b\tan\frac{b}{2}) & ,\: a\geq0\\
(1+a)b\cot b-ab\cot\frac{b}{2} & ,\: a<0\end{cases}\\
 & \geq\psi(ib_{0})-\frac{\epsilon}{b_{0}}\psi(ib_{0}).\end{alignat*}
The use of lemma \ref{lem:main estimate} with this $b$ yields:\begin{alignat}{1}
\left|h(x,\xi,u)\right| & \lesssim\left(\frac{b}{\sin b}\right)^{\alpha}e^{-\psi(ib_{0})R^{2}-2b_{0}u+\frac{\epsilon}{b_{0}}(\psi(ib_{0})+2b_{0}u)}\label{eq:h leq epsilon}\\
 & \lesssim\left(\frac{1}{\epsilon}\right)^{\alpha}e^{-(1-\frac{\epsilon}{b_{0}})d_{CC}(\zeta,\eta)^{2}},\nonumber \end{alignat}
first for $b_{0}\neq\pi$. But, if $b_{0}=\pi$, then we have $x=\xi=0$
or $x=-\xi$. We claim that in this case estimate \eqref{eq:h leq epsilon}
remains true: If $x=\xi=0$ then \begin{alignat*}{1}
\left|h(x,\xi,u)\right| & \lesssim\left(\frac{b}{\sin b}\right)^{\alpha}e^{-2b_{0}u+2b_{0}u\frac{\epsilon}{b_{0}}}\\
 & \lesssim\left(\frac{1}{\epsilon}\right)^{\alpha}e^{-(1-\frac{\epsilon}{b_{0}})d_{CC}(\zeta,\eta)^{2}},\end{alignat*}
since $2b_{0}u=2\pi u=d_{CC}(\zeta,\eta)^{2}$. And, if $x=-\xi$,
then $a=-1$ and $\psi(ib)=b\cot\frac{b}{2}\geq0$, so that\begin{alignat*}{1}
\left|h(x,\xi,u)\right| & \lesssim\left(\frac{b}{\sin b}\right)^{\alpha}e^{-\psi(ib)R^{2}-2b_{0}u+2b_{0}u\frac{\epsilon}{b_{0}}}\\
 & \lesssim\left(\frac{1}{\epsilon}\right)^{\alpha}e^{-2b_{0}u+2b_{0}u\frac{\epsilon}{b_{0}}}\\
 & =\left(\frac{1}{\epsilon}\right)^{\alpha}e^{-(1-\frac{\epsilon}{b_{0}})d_{CC}(\zeta,\eta)^{2}},\end{alignat*}
since again $2b_{0}u=2\pi u=d_{CC}(\zeta,\eta)^{2}.$

Setting $\epsilon=\frac{1}{1+d_{CC}(\zeta,\eta)^{2}}$, which optimizes
this inequality (up to a constant), gives us the desired result:\begin{equation}
\left|h(x,\xi,u)\right|\lesssim(1+d_{CC}(\zeta,\eta)^{2})^{\alpha}e^{-d_{CC}(\zeta,\eta)^{2}}.\end{equation}

\end{proof}
This gives the following main result:

\begin{thm}
For $\zeta=(x,0),\;\eta=(\xi,u)\in\mathbb{R}^{n+1}$ we have\begin{equation}
\left|K_{t}(x,\xi,u)\right|\lesssim t^{-\frac{n}{2}-1}\min\left(1+\frac{d_{CC}(\zeta,\eta)}{\left|x+\xi\right|},1+\frac{d_{CC}(\zeta,\eta)^{2}}{4t}\right)^{\alpha}e^{-\frac{1}{4t}d_{CC}(\zeta,\eta)^{2}},\end{equation}
with $\alpha=\max\left(\frac{n}{2}-1,0\right)$.
\end{thm}
\begin{proof}
The proof is almost trivial, if one uses \eqref{eq:h eq kt} and the
homogeneity of $d_{CC}$.
\end{proof}
\nocite{Beals2000,Calin2005,Calin2005a,Folland1989,Greiner2002,Meyer2006,Stein1993,Strichartz1986,Strichartz1989,Sikora2004}\bibliographystyle{plain}
\bibliography{refs}

\begin{thebibliography}{10}

\bibitem{Beals2000}
Richard Beals, Bernard Gaveau, and Peter~C. Greiner.
\newblock Hamilton-{J}acobi theory and the heat kernel on {H}eisenberg groups.
\newblock {\em J. Math. Pures Appl. (9)}, 79(7):633--689, 2000.

\bibitem{Calin2005a}
Ovidiu Calin and Der-Chen Chang.
\newblock The geometry on a step 3 {G}rushin model.
\newblock {\em Appl. Anal.}, 84(2):111--129, 2005.

\bibitem{Calin2005}
Ovidiu Calin, Der-Chen Chang, Peter Greiner, and Yakar Kannai.
\newblock On the geometry induced by a {G}rusin operator.
\newblock In {\em Complex analysis and dynamical systems II}, volume 382 of
  {\em Contemp. Math.}, pages 89--111. Amer. Math. Soc., Providence, RI, 2005.

\bibitem{Chow1939}
W.L. Chow.
\newblock {\"U}ber {S}ysteme von {L}inearen {P}artiellen
  {D}ifferentialgleichungen erster {O}rdnung.
\newblock {\em Math. Ann.}, 117:98--105, 1939.

\bibitem{Folland1989}
Gerald~B. Folland.
\newblock {\em Harmonic analysis in phase space}, volume 122 of {\em Annals of
  Mathematics Studies}.
\newblock Princeton University Press, Princeton, NJ, 1989.

\bibitem{Greiner2002}
Peter~C. Greiner, David Holcman, and Yakar Kannai.
\newblock Wave kernels related to second-order operators.
\newblock {\em Duke Math. J.}, 114(2):329--386, 2002.

\bibitem{Meyer2006}
Ralf Meyer.
\newblock {\em $L^p$-estimates for the wave equation associated to the Grusin
  operator}.
\newblock PhD thesis, Kiel, Univ., 2006.

\bibitem{Sikora2004}
Adam Sikora.
\newblock Riesz transform, {G}aussian bounds and the method of wave equation.
\newblock {\em Math. Z.}, 247(3):643--662, 2004.

\bibitem{Stein1993}
Elias~M. Stein.
\newblock {\em Harmonic analysis: real-variable methods, orthogonality, and
  oscillatory integrals}, volume~43 of {\em Princeton Mathematical Series}.
\newblock Princeton University Press, Princeton, NJ, 1993.
\newblock With the assistance of Timothy S. Murphy, Monographs in Harmonic
  Analysis, III.

\bibitem{Strichartz1986}
Robert~S. Strichartz.
\newblock Sub-{R}iemannian geometry.
\newblock {\em J. Differential Geom.}, 24(2):221--263, 1986.

\bibitem{Strichartz1989}
Robert~S. Strichartz.
\newblock Corrections to: ``{S}ub-{R}iemannian geometry'' [{J}. {D}ifferential
  {G}eom.\ {\bf 24} (1986), no.\ 2, 221--263; {MR}0862049 (88b:53055)].
\newblock {\em J. Differential Geom.}, 30(2):595--596, 1989.

\end{thebibliography}

\end{document}